\documentclass[a4paper,11pt,reqno]{amsart}

\usepackage[T1]{fontenc}
\usepackage{lmodern}
\usepackage{microtype}

\usepackage{mathtools,amssymb,amsthm}
\usepackage{mathrsfs,esint}
\usepackage{enumitem}
\usepackage{geometry}

\usepackage{xcolor}
\usepackage{tikz}

\usepackage{mleftright}
\mleftright % redefine \left as \mleft and \right as \mright.

\usepackage[backend=biber,
style=alphabetic,
sorting=nyt,
maxbibnames=10,
maxcitenames=10,
giveninits=true,
doi=false,
url=false]{biblatex}

\renewbibmacro{in:}{%
  \ifentrytype{article}{}{\printtext{\bibstring{in}\intitlepunct}}}

\AtBeginBibliography{%
}

\usepackage{hyperref}

\hypersetup{
  colorlinks=true,
  linkcolor=blue,      % internal links: sections, theorems, equations
  citecolor=blue,      % \cite
  urlcolor=blue,       % \url, \href
  filecolor=blue,
  breaklinks=true,     % allow links to break across lines
  bookmarksnumbered=true,
  pdftitle={Optimal regularity for vectorial, higher order and non-minimizing Bernoulli problems},
  pdfauthor={De Philippis, Hirsch, Nahon},
}

\newcommand \Om{\Omega}

\newcommand \eps{\epsilon}

\newcommand \B{\mathbb{B}}

\newcommand \ov{\overline}
\newcommand \M{\mathscr{M}}
\newcommand \N{\mathbb{N}}
\newcommand \Z{\mathbb{Z}}
\newcommand \C{\mathbb{C}}
\newcommand \R{\mathbb{R}}
\newcommand \loc{\mathrm{loc}}

\newcommand \Normr[1]{\left\Vert #1\right\Vert_{L^2(\B_1\setminus\B_{\frac{1}{2}})}}
\newcommand \Norm[1]{\left\Vert #1\right\Vert_{L^2(\B_1)}}
\newcommand \Ho[1]{\mathscr{H}_{#1}}

\newcommand{\iu}{\mathrm{i}}             % Imaginary unit

\numberwithin{equation}{section}

\newtheorem{theorem}{Theorem}
\newtheorem{corollary}[theorem]{Corollary}
\newtheorem{lemma} [theorem]{Lemma}
\newtheorem{proposition} [theorem]{Proposition}

\theoremstyle{definition}
\newtheorem{definition}[theorem] {Definition}

\newtheorem{remark} [theorem]{Remark}

\begin{document}

\title[Optimal Regularity for Bernoulli problems]{Optimal regularity for vectorial, higher order and non-minimizing Bernoulli problems}

\author[G.De Philippis]{Guido De Philippis}
\address{Dipartimento di Matematica ``Tullio Levi-Civita'', Via Trieste 63, Torre Archimede, Padova, Italy}
\email{guido@math.unipd.it}

\author[J.Hirsch]{Jonas Hirsch}
\address{Universität Leipzig, Mathematisches Institut, Augustusplatz 10, 04109 Leipzig, Germany}
\email{hirsch.jonas@math.uni-leipzig.de}

\author[M.Nahon]{Mickaël Nahon}
\address{Univ. Grenoble Alpes, CNRS, Grenoble INP, LJK, 38000 Grenoble, France}
\email{mickael.nahon@univ-grenoble-alpes.fr}



\begin{abstract}
We prove the Lipschitz regularity of solutions for a wide class of generalized Bernoulli free boundary problems, in vectorial, higher order, and non-minimizing settings. Our method does not rely on notions of viscosity solutions or comparison methods, which allows us to reach the optimal regularity for Bernoulli-type problems associated to elliptic operators which do not satisfy any maximum principle, namely the elasticity, biharmonic and Stokes equation. With a similar method we obtain the optimal Lipschitz regularity for stationary, non-minimizing solutions of the standard Bernoulli problem in two dimensions.
% In a second stage we study the non-degeneracy properties of minimizers: we prove that the non-degeneracy property of solutions does not hold without some additional hypotheses on the functional.
\end{abstract}

\maketitle

\tableofcontents

\section{Introduction}

\subsection{Background}

Several questions in pure and applied mathematics lead to the study of local minimizers and of
stationary solutions of Bernoulli-type energies
\begin{equation}\label{e:bernoullintro}
\int_{\B_1} Q(\nabla^nu)+\chi_{\nabla^m u\ne 0},
\end{equation}
where $\B_1$ is the  unit ball of  $\R^d$, \(u\in H^n(\B_1,\R^p)\), \(m\le n\), and \(\nabla^\ell u\in
(\R^d)^{\otimes \ell}\otimes \R^p\) denotes the vector of the \(\ell\)-th derivatives of
\(u\). Here we say $u\in H^n(\B_1,\R^p)$ is a \textit{local minimizer} of a functional $E:H^n(\B_1,\R^p)\to \R$ when $E(u)\leq E(u+v)$ for all $v\in H^n_0(\B_1,\R^p)$, and a \textit{stationary solution} when $\left.\frac{d}{dt}\right|_{t=0}E(u\circ (\mathrm{id}+t\xi))=0$ for all $\xi\in\mathcal{C}^\infty_c(\B_1,\R^d)$.

When \(n=1\), \(p=1\), and \(Q\) is the square of the norm, the functional reduces to the
Bernoulli energy
\begin{equation}\label{eq_AC}
\int_{\B_1} |\nabla u|^2+\chi_{u\ne 0}.
\end{equation}
If one restricts the competitors to non-negative functions, one obtains the so-called
\emph{one-phase Bernoulli problem}, whose optimality conditions are formally given by the
overdetermined system
\begin{equation*}\label{eq_Bernoulli}
\begin{cases}
\Delta u =0 &\text{ in }\B_1\cap\{u>0\},\\
u =0 &\text{ in }\B_1\cap\partial\{u>0\},\\
\partial_\nu u =1 &\text{ in }\B_1\cap\partial\{u>0\},
\end{cases}
\end{equation*}
where $\partial_\nu$ is the inward normal derivative. The study of this overdetermined problem goes back to~\cite{Beurling}, and it arises
naturally in several questions related to fluid dynamics, shape optimization, capillarity surfaces, and potential theory.

The regularity of local minimizers of~\eqref{eq_AC} was first studied in the foundational
work~\cite{AC81}, where the authors show that minimizers are locally Lipschitz, that they
have at least linear growth near every boundary point, and that every sufficiently flat
part of the boundary is analytic. Subsequent works~\cite{W99,CJK01,DJ09,JS15} show that the
free boundary \(\partial\{u>0\}\) is the union of an embedded analytic hypersurface and of
a singular set of Hausdorff dimension at most \(d-d^*\), where \(d^*\in\{5,6,7\}\) is a
critical dimension characterized by the existence of non-flat homogeneous solutions in
\(\R^{d^*}\).

If one does not assume that \(u\) has constant sign, one obtains instead a special case of
the two-phase free boundary problem, sometimes called the \emph{symmetric} two-phase
problem. Here as well, solutions are locally Lipschitz with at least linear growth at the
boundary~\cite{ACF84}, and the free boundary was described in the subsequent series of
papers~\cite{SV19,DSV21,DSV25,FV26}.

Note that establishing the optimal Lipschitz regularity is the first key ingredient in establishing the free boundary regularity. Indeed   it allows for a blow-up analysis around a free boundary point \(z\) by studying the limit points of the  sequence \(r^{-1}u(z+r\,\cdot\,)\).

Other shape optimization problems lead naturally to more general functionals of the form~\eqref{e:bernoullintro}. For instance:
\begin{itemize}
\item In the study of P{\'o}lya and Szeg{\"o}'s conjecture on the optimal shape for the critical buckling load of a clamped plate~\cite{AshbBucu03}, as well as the optimal shape for its fundamental eigenvalue~\cite{AB95,N95}, one is led to consider local minimizers of
\begin{equation}\label{eq:buckledplate}
\int_{\B_1} |\nabla^2u|^2+\chi_{u\ne 0}.
\end{equation}
See also the introduction of~\cite{LN25} for a more detailed account.
\item The study of optimal profiles in viscous flows~\cite{P73,R95} leads naturally to
vectorial problems with a \emph{differential constraint}, namely to the local minimizers of
\begin{equation}\label{eq:stokes}
\int_{\B_1} |\nabla u |^2+\chi_{u\ne 0}
\end{equation}
in the class of \emph{divergence-free} vector fields \(H^1_{\textup{div}}(\B_1,\R^d)\).
In dimension \(2\) the introduction of the stream function \(u=\nabla^\perp \psi\) reduces this
problem to minimizers of
\begin{equation*}\label{eq:streamfunctionminimizers}
\int_{\B_1} |\nabla^2\psi|^2+\chi_{\nabla\psi\ne 0}.
\end{equation*}
\item The study of the optimal domain for the first Lamé eigenvalue and for the
Korn--Poincaré constant leads to energies of the form
\begin{equation}\label{eq_lame}
\int_{\B_1} \mathbb{A}e(u):e(u)+\chi_{u\neq 0},
\end{equation}
where $u\in H^1(\B_1,\R^d)$,
\[
\mathbb{A}S=2\mu S+\lambda\,\mathrm{Tr}(S)I_d,
\qquad e(u)=\frac{\nabla u+\nabla u^T}{2};
\]
for some Lamé parameters $(\lambda,\mu)\in\R^2$ such that $\mu>0$, $\lambda+2\mu>0$, see~\cite{HLP24} and~\cite{F25}.

\item When one minimizes the Bernoulli energy in a restricted class, for instance in the class of simply
connected domains, one often has to work with \emph{stationary solutions} with respect to
inner variations, namely with functions that satisfy
\begin{equation}\label{eq:stationary}
\left.\frac{d}{dt}\right|_{t=0}\int_{\B_1}\left(|\nabla(u\circ (\mathrm{id}+t\xi))|^2
+\chi_{u\circ (\mathrm{id}+t\xi)\ne 0}\right)=0
\end{equation}
for all \(\xi\in\mathcal{C}^\infty_c(\B_1,\R^d)\). Stationary solutions have recently been
studied in~\cite{KW25} under the \emph{a priori} assumption of Lipschitz regularity.
\end{itemize}

Very little is known for the minimizers of these problems, neither for the regularity of
the functions nor for that of the free boundaries. To the best of our knowledge, the best
results currently available are the following.
\begin{itemize}
\item Minimizers of~\eqref{eq_lame} for $\mu>0$, $\lambda+\mu>0$ belong to $\mathcal{C}^\alpha_\loc(\B_1,\R^d)$ for every $\alpha\in (0,1)$ by  \cite{HLP24} in dimension $2$, $3$ and \cite{F25} in any dimension.

\item Fourth-order problems in the spirit of~\eqref{eq:buckledplate} were studied in~\cite{DipiKaraVald21}, where several results are proved \emph{conditionally on} the
\(C^{1,1}\) regularity of the minimizer, and in \cite{M22}. In~\cite{LN25}, the third author and Lamboley gave
a fine description of the free boundary points of two-dimensional minimizers of~\eqref{eq:buckledplate}. In particular, they classified completely the free boundary points
around which the solution is locally \(C^{1,1}\), but they left open whether other types of
free boundary point can occur; see also~\cite{GM24,GM26} where a partial classification of global minimizers is given.

\item In~\cite{F25}, minimizers of~\eqref{e:bernoullintro} were shown to belong to
\(C^{n-1,\alpha}\) for every \(\alpha\in(0,1)\), provided \(Q\) is elliptic. The same paper
shows that minimizers of~\eqref{eq:stokes} belong to \(C^{\alpha}\) for every
\(\alpha\in(0,1)\).
\end{itemize}

In particular, it was not known whether local minimizers of~\eqref{e:bernoullintro} are of class
\(W^{n,\infty}_\loc\).  Note  that this is  the optimal expected regularity, since \(\nabla^nu\) should
jump across the free boundary. Furthermore, as explained in the next section, this is the critical regularity that allows to  consider blow-up limits,  which is the first step  in the study of free boundary regularity.

To the best of our knowledge, every proof of the optimal regularity for the classical
Bernoulli problem~\eqref{eq_AC} uses one of the following two facts.
\begin{itemize}
\item For the one-phase problem, the sign of the solution together with some form of the
maximum principle; see~\cite[Ch.~3]{V23} for a review of the known techniques.
\item For the two-phase and the vectorial problems, the Alt--Caffarelli--Friedman
monotonicity formula~\cite[Th.~5.3]{ACF84}; see~\cite{BHM,DavidToro15,BMPV}.
\end{itemize}
Neither ingredient is available for the functionals~\eqref{e:bernoullintro}. Solutions of high order equations  satisfy no maximum principle and no analogue of the monotonicity formula is known  when the functional is not the Dirichlet energy. Therefore, in order to prove optimal regularity one has to devise a completely different strategy.

\subsection{Main results}

Our first main result is the proof of  the optimal regularity for  local minimizers of~\eqref{e:bernoullintro}  and  of the
Stokes--Bernoulli problem~\eqref{eq:stokes}. We also show optimal regularity in \(2\)-d  for the stationary
solutions of~\eqref{eq:stationary}. All these results are consequences of the abstract
Theorem~\ref{th_main} below, which we now state.

For \(p\in\N^*\) and \(\phi:\B_1\to\R^p\) measurable, we define the associated
\emph{boundary} as
\begin{equation*}\label{eq_Fu}
\mathcal{F}_\phi=\left\{z\in\B_1:\forall r\in (0,1-|z|),\
|\B_{z,r}\cap\{\phi=0\}|\cdot|\B_{z,r}\cap\{\phi\neq 0\}|>0\right\}
\end{equation*}
where $\B_{z,r}:=z+r\B_1$.
\begin{theorem}\label{th_main}
For any \(K\geq 0\) and \(d\in\N^*\) there exist \(C,\alpha>0\) such that the following
holds. Let \(\phi\in L^2(\B_1,\R^p)\) for some \(p\in\N^*\), assume that
\(\mathcal{F}_\phi\cap\B_1\neq \emptyset\), and assume that for every ball
\(\B_{z,r}\subset\B_1\) there exists \(\phi_{\B_{z,r}}\in L^2(\B_{z,r},\R^p)\) such that
\begin{equation}\label{eq_hyp1}
\int_{\B_{z,r}}|\phi-\phi_{\B_{z,r}}|^2\leq |\B_{z,r}\cap\{\phi=0\}|
\end{equation}
and
\begin{equation}\label{eq_hyp2}
r\Vert \nabla \phi_{\B_{z,r}}\Vert_{L^\infty(\B_{z,\frac{1}{2}r})}
\leq K\left(\fint_{\B_{z,r}}\left|\phi_{\B_{z,r}}
-\fint_{\B_{z,r}}\phi_{\B_{z,r}}\right|^2\right)^\frac{1}{2}.
\end{equation}
Then for almost every \(z\in\B_{\frac{1}{2}}\),
\[|\phi(z)|\leq C\left(1+\mathrm{dist}(z,\mathcal{F}_\phi)^\alpha
\Vert \phi\Vert_{L^2(\B_1)}\right).\]
\end{theorem}
 In the applications \(\phi\) is
the highest derivative of the minimizer and \(\phi_{\B_{z,r}}\) is the corresponding
derivative of the solution of the associated linear equation with the same boundary data on
\(\B_{z,r}\). In this case~\eqref{eq_hyp1} follows from energy comparison and~\eqref{eq_hyp2} from
interior elliptic estimates.

\subsubsection{\(W^{n,\infty}\) regularity of minimizers}

% As a corollary of Theorem~\ref{th_main}, we obtain the optimal regularity for a wide range
% of Bernoulli-type free boundary problems.

Recall that a quadratic form \(Q\) on
\((\R^d)^{\otimes n}\otimes \R^p\) is \emph{elliptic} if it satisfies the
Legendre--Hadamard condition: there exists \(\lambda_Q>0\) such that
\begin{equation}\label{eq_legendrehadamard}
\forall\xi\in\R^d,\ \forall v\in\R^p,\quad
Q(\xi^{\otimes n}\otimes v)\geq\lambda_Q|\xi|^{2n}|v|^2.
\end{equation}
Taking the Fourier transform on both sides, Gårding's inequality gives
\[
\forall u\in H^n(\R^d,\R^p),\quad
\int_{\R^d}Q(\nabla^n u)\geq\lambda_Q\int_{\R^d}|\nabla^n u|^2.
\]

\begin{corollary}\label{cor_Q}
Let \(d,p,n\in\N^*\) and let \(Q\) be an elliptic quadratic form on \((\R^d)^{\otimes
n}\otimes \R^p\). Let \(m\in\{0,1,\hdots,n-1\}\) and let \(u\in H^n(\B_1,\R^p)\) be a local
minimizer of
\[v\in H^n(\B_1,\R^p)\mapsto \int_{\B_1}\left(Q(\nabla^n v)+\chi_{\nabla^m v\neq 0}\right).\]
Then \(u\in W^{n,\infty}_\loc(\B_1,\R^p)\). More precisely, if \(\mathcal{F}_{\nabla^m
u}\cap\B_1\neq\emptyset\), then for almost every \(z\in\B_{\frac{1}{2}}\),
\[|\nabla^n u(z)|\leq C\left(1+\mathrm{dist}(z,\mathcal{F}_{\nabla^m u})^\alpha
\Vert \nabla^n u\Vert_{L^2(\B_1)}\right)\]
for some constants \(C,\alpha>0\) depending only on \(d,p,n,m,Q\).
\end{corollary}

\begin{remark}
We do not know whether the stronger estimate
\[\mathcal{F}_u\cap\B_1\neq\emptyset\Rightarrow
\Vert\nabla^n u\Vert_{L^\infty(\B_{\frac{1}{2}})}\leq C\]
holds for a constant \(C>0\) independent of \(u\). We do have, however, the following
uniform bound. Let \(u\) be as in Corollary~\ref{cor_Q}, let \(z\in\mathcal{F}_{\nabla^m
u}\), \(r>0\), and let
\[u_{z,r}(\zeta)=r^{-n}\left(u(z+r\zeta)-\sum_{\ell=0}^{m-1}\frac{r^\ell}{\ell !}
\nabla^\ell u(z):\zeta^{\otimes \ell}\right)\]
be the normalized blow-up sequence of \(u\) at \(z\). Then \(\limsup_{r\to 0}\Vert
u_{z,r}\Vert_{W^{n,\infty}(\B_1,\R^p)}\) is bounded by a constant depending on
\(d,p,n,m,Q\) but not on \(z\) and \(u\). In other words, the blow-up limits of any
minimizer at any non-trivial boundary point form a bounded set in
\(W^{n,\infty}(\B_1,\R^p)\).
\end{remark}

\begin{remark}
Corollary~\ref{cor_Q} gives a new proof of the Lipschitz regularity of the solutions of the
one-phase, two-phase and vectorial Bernoulli problems. Our proof does not seem to extend,
however, to the local minimizers of the general two-phase problem
\[u\in H^1(\B_1,\R)\mapsto \int_{\B_1}\left(|\nabla u|^2+\lambda^+ \chi_{u> 0}
+\lambda^-\chi_{u<0}\right)\]
when \(\lambda^+\neq \lambda^-\).
\end{remark}

\subsubsection{Non-degeneracy}

Together with the optimal regularity, a key property of the minimizers of~\eqref{eq_AC} is
the so-called \emph{non-degeneracy}, namely that
\[
|u(z)|\gtrsim\mathrm{dist}(z,\mathcal{F}_u).
\]
Combined with the Lipschitz bound, it guarantees that the blow-up limits are not zero, and
therefore that they carry information on the free boundary.

For the minimizers of~\eqref{e:bernoullintro} the non-degeneracy is more subtle. It fails
when \(m\in\{1,\hdots,n-1\}\), as we discuss in Remark~\ref{rk_nondeg} and Proposition~\ref{prop_pipe}. When \(m=0\), on the other hand, it holds.

\begin{proposition}\label{prop_nondegeneracy}
Let \(u\) satisfy the hypotheses of Corollary~\ref{cor_Q} with \(m=0\), and assume
\(u(0)\neq 0\). Then
\[\Vert u\Vert_{L^2(\B_1)}\geq c,\]
where \(c=c(d,p,n,Q)>0\) does not depend on \(u\).
\end{proposition}

\subsubsection{Optimal regularity for the Stokes problem}

Our techniques also give the optimal regularity for constrained vectorial free boundary
problems. For simplicity we treat the case of divergence-free vector fields. For an open
set \(\Om\subset\R^d\) we set
\[H^1_{\mathrm{div}}(\Om,\R^d)=\left\{u\in H^1(\Om,\R^d):\mathrm{div}(u)=0\right\},
\qquad H^1_{0,\mathrm{div}}(\Om,\R^d)=H^1_{\mathrm{div}}(\Om,\R^d)\cap H^1_{0}(\Om,\R^d).\]

\begin{corollary}\label{cor_contrainte}
Let \(d\geq 2\) and let \(u\in H^1_{\mathrm{div}}(\B_1,\R^d)\) be a local minimizer of
\[v\in H^1_{\mathrm{div}}(\B_1,\R^d)\mapsto \int_{\B_1}\left(|\nabla v|^2
+\chi_{v\neq 0}\right)\]
such that \(\mathcal{F}_u\cap\B_1\neq \emptyset\). Then for almost every
\(z\in\B_{\frac{1}{2}}\),
\[|\nabla u(z)|\leq C\left(1+\mathrm{dist}\left(z,\mathcal{F}_u\right)^\alpha
\Vert \nabla u\Vert_{L^2(\B_1)}\right)\]
for some constants \(C,\alpha>0\) depending only on \(d\).
\end{corollary}

As a direct application to a well-studied shape
optimization problem, define the drag of an object $K$ (a compact subset of \(\R^d\) where \(d\geq 3\)) in the
direction \(u_\infty\in\R^d\setminus\{0\}\) in a Stokes fluid as
\begin{equation}\label{eq_drag}
\mathcal{D}(K):=\inf\left\{\int_{\R^d\setminus K}|\nabla u|^2,\
u\in u_\infty+\dot{H}^1_{\mathrm{div}}(\R^d,\R^d): u|_{K}\equiv 0\right\}.
\end{equation}
where $u|_{K}\equiv 0$ is understood as the quasi-continuous representative of $u$ vanishing quasi-everywhere in $K$ i.e. $u\in H^1_{0,\loc}(\R^d\setminus K,\R^d)$, and $\dot{H}^1_{0,\mathrm{div}}(\R^d,\R^d)$ is the closure of divergence-free compactly supported vector fields under the homogeneous $H^1$ norm. The minimization of \(\mathcal{D}(K)\) under a volume constraint on \(K\) was studied in~\cite{P73,B74} among three-dimensional shapes that are rotationally invariant around
\(u_\infty\). The conjectured minimizer is smooth away from its rotation axis, where it has
two angular points of opening \(\frac{2\pi}{3}\). For any minimizer of \(\mathcal{D}\)
among closed sets \(K\) of given volume, with associated flow \(u\in
u_\infty+\dot{H}^1_{\mathrm{div}}(\R^d,\R^d)\), we prove that
\[\Vert\nabla u\Vert_{L^\infty(\R^d,\R^d\otimes\R^d)}<\infty;\]
see Proposition~\ref{prop_drag}. Any precise description of \(\partial K\) remains open.

\subsubsection{Lipschitz regularity of stationary solutions}

As a further corollary of Theorem~\ref{th_main}, we obtain the Lipschitz regularity of
two-dimensional solutions of the classical scalar Bernoulli problem in a weak sense.

We remind that \(u\in H^1(\B_1,\R)\) is a stationary solution of the Bernoulli problem if it
satisfies~\eqref{eq:stationary}, or equivalently if
\begin{equation*}\label{eq:critical_scalar}
 \forall \xi\in\mathcal{C}^\infty_c(\B_1,\R^d),\quad
 \int_{\B_1}\left(|\nabla u|^2\mathrm{div}\,\xi-2\nabla u\cdot D\xi\,\nabla u\right)
 =\int_{\B_1}\chi_{u=0}\,\mathrm{div}\,\xi.
\end{equation*}
Note that we do not require \(u\) to satisfy an equation on its support.

This problem was studied extensively in~\cite{KW25}, see also~\cite{KW25_wave} for an
application to the construction of water wave solutions. There the Lipschitz regularity was
conjectured to fail, see~\cite[Rem.~3.1]{KW25}, and was assumed as a hypothesis instead.

Classical solutions of this problem, meaning solutions whose free boundary is assumed
\emph{a priori} to be smooth, were studied in~\cite{JK16,JK19}, in relation with the
non-trivial global solutions constructed in~\cite{HHP11}; see also~\cite{K24,JK24} for
classical solutions satisfying an additional index bound. The Lipschitz regularity of
\emph{inner variation} solutions of elliptic equations coming from harmonic maps and
nonlinear elasticity was studied in~\cite{CIKO14}, see also~\cite{IKO13}, using the
regularity theory of quasiconformal mappings.

\begin{corollary}\label{cor_lipschitz_crit}
Let \(d=2\). There exist \(C,\beta>0\) such that the following holds. Let \(u\in
H^1(\B_1,\R)\) be a stationary solution of the Bernoulli problem with
\(\mathcal{F}_u\cap\B_1\neq\emptyset\). Then for almost every \(z\in \B_\frac{1}{2}\),
\[|\nabla u(z)|\leq C\left(1+\mathrm{dist}(z,\mathcal{F}_u)^\beta
\Vert\nabla u\Vert_{L^2(\B_1)}\right).\]
\end{corollary}

%A close inspection of the proof gives \(\beta=1\), but we do not try to optimize this exponent, since we do not know whether the stronger estimate \(\Vert \nabla u\Vert_{L^\infty(\B_\frac{1}{2})}\leq C\) holds.

As a consequence of Corollary~\ref{cor_lipschitz_crit}, the result of~\cite[Th.~1.2]{KW25}
holds in dimension two under the sole hypothesis that \(u\) satisfies~\eqref{eq:stationary} and belongs to \(\mathcal{C}_\loc^2(\{u\neq 0\})\).

In general, when considering vectorial stationary solutions, i.e. functions \(u \in H^1(\B_1,\mathbb{R}^{p})\) such that:
\begin{equation}\label{eq:critical_vectorial}
 \forall \xi\in\mathcal{C}^\infty_c(\B_1,\R^d),\quad
 \int_{\B_1}\left(|\nabla u|^2\mathrm{div}\,\xi-2\sum_{j=1}^p\nabla u_j\cdot D\xi\,\nabla u_j\right)
 =\int_{\B_1}\chi_{u=0}\,\mathrm{div}\,\xi,
\end{equation}
the analogue of Corollary~\ref{cor_lipschitz_crit} is false, see Remark~\ref{rem_counterexample_crit}. However we can show Lipschitz regularity under a slightly stronger assumption, that is satisfied for instance in the relevant situation in which $u$ is harmonic in its support in the weak sense. Namely we consider  the following   condition:
\begin{equation}\label{eq_structure}
\forall k,l\in\{1,\hdots,p\},\ u_k\Delta u_l=0,
\end{equation}
which should be interpreted in the weak sense:
\begin{equation}\label{eq_structure_weak}
\forall k,l\in\{1,\hdots,p\},\forall \varphi\in\mathcal{C}^\infty_c(\B_1,\R),\  \int_{\B_1}\nabla(\varphi u_k)\cdot\nabla u_l=0.
\end{equation}
Then our second main result is the following.
\begin{theorem}\label{th_crit_vectorial_intro}
 Let $d=2$. Let \(u\in H^1(\B_1,\mathbb{R}^p)\) be a stationary solution of the Bernoulli problem, i.e. a solution of~\eqref{eq:critical_vectorial}, that satisfies~\eqref{eq_structure_weak}. Then

 \[
  \Vert\nabla u\Vert_{L^\infty(\B_\frac{1}{2})}\leq C\left(1+\Vert u\Vert_{L^2(\B_1)}\right)
  \]
for some constant $C>0$ that only depends on $p$.
\end{theorem}

Consequences of the above results  in shape optimization and for overdetermined equations
are discussed in Section~\ref{sec_spectral}; see for instance Proposition~\ref{prop_overdetermined}. Furthermore in  Section~\ref{sec_proof_main_crit} we also give an alternative proof of the non-degeneracy of the support obtained in~\cite[Lem.~9.2]{KW25}.

\subsection{Idea of the proof}

As explained above, our regularity results are consequences of the abstract Theorem~\ref{th_main}, whose proof has two steps.
\begin{itemize}
\item[--] First we show that any map satisfying the assumptions of Theorem~\ref{th_main} is
of class \(BMO\).
\item[--] By the previous step, in order to bound \(\phi\) it is enough to bound its
average. We do this by a dichotomy: either the average is bounded, or the density of
\(\{\phi=0\}\) in \(\B_{1/2}\) is at most a fixed fraction of its density in \(\B_1\). The
second alternative cannot occur infinitely often at a Lebesgue point of \(\{\phi=0\}\),
which gives the desired bound. We note here that  similar dichotomy arguments were used in~\cite{FigaShah14}
and~\cite{DeSSavi20}.
\end{itemize}

Once Theorem~\ref{th_main} is established, the applications to minimizers follow by taking
\(\phi=\nabla^nu\) and \(\phi_{\B_{z,r}}=\nabla^n v_{\B_{z,r}}\), where \(v_{\B_{z,r}}\)
minimizes the free energy \(\int Q(\nabla^nv)\) among the functions with the same boundary
values as \(u\) on \(\partial \B_{z,r}\). With this choice,~\eqref{eq_hyp1} follows from
minimality and~\eqref{eq_hyp2} from the regularity estimates for \(v_{\B_{z,r}}\).

For Corollary~\ref{cor_lipschitz_crit} we cannot rely on minimality. The first variation
gives instead
\[
4\partial_{\bar z}(\partial_z u)^2-\partial_z\chi_{u\ne 0}=0,
\]
where \(\partial_{z}=\frac{1}{2}(\partial_x-\iu \partial_{y})\) and \(\partial_{\ov{z}}=\frac{1}{2}(\partial_x+\iu \partial_{y})\) are the complex derivatives. Introducing the
Beurling transform \(\mathcal{B}= \partial_{\bar z}^{-1}\circ \partial_{z}\), one checks that
\[
h=4(\partial_z u)^2-\mathcal{B}\chi_{u\ne 0}
\]
is holomorphic. We then follow the strategy above with \(\phi=(\partial_z u)^2\) and with a
local version of \(h\) as \(\phi_{\B_{z,r}}\).

The proof of Theorem~\ref{th_crit_vectorial_intro} is instead based on a different route, inspired by the regularity theory for minimal surfaces. We first show that for any \(2\)-convex function \(Q\) such that $\nabla  Q(0)=0$, the function \(C|z|^2+Q\circ u(z)\) is subharmonic for some sufficiently large $C$. This allows to prove Hölder regularity of the solution. Once this is proved we run various contradiction and compactness arguments to show that  either the norm of the gradient decays or the solution is close, in an annulus, to a linear conformal map. The final critical argument is to show that in the latter case the information propagates at smaller scales, eventually showing that the function can only vanish at the center of the ball, from which regularity easily follows. See Section~\ref{sec_proof_main_crit} for a more detailed account of the proof.

\subsection{Notation and organization of the paper}

In Section~\ref{sec_proof_th_main} we prove Theorem~\ref{th_main}. In Section~\ref{sec_vectorialBernoulli} we give the
applications of Theorem~\ref{th_main} to Bernoulli-type problems, namely Corollaries~\ref{cor_Q} and~\ref{cor_contrainte} and Proposition~\ref{prop_nondegeneracy}. In Section~\ref{sec_proof_main_crit} we deduce Corollary~\ref{cor_lipschitz_crit} from Theorem~\ref{th_main}, establish Theorem \ref{th_crit_vectorial_intro}, and study the non-degeneracy of stationary solutions. Finally, in Section~\ref{sec_spectral} we show how our results extend to several shape optimization problems
coming from spectral inequalities.

We write
\[a\lesssim b\]
when \(a\leq Cb\), where:
\begin{itemize}
\item in Section~\ref{sec_proof_th_main}, \(C>0\) depends only on \(d\) and \(K\);
\item in Sections~\ref{sec_proof_main_crit} and~\ref{sec_spectral}, \(C>0\) is universal.
\end{itemize}
For \(f\in L^2(\B_{r})\), possibly vector-valued, we define its deviation in \(\B_r\) as
\[ D(f,r)=\sqrt{\fint_{\B_r} \left|f-\fint_{\B_r}f\right|^2}.\]

\section{Proof of Theorem~\ref{th_main}}\label{sec_proof_th_main}
In all this section, we fix \(d,p\geq 1\), \(K\geq 0\). We let \(\M\) be the set of functions \(\phi\in L^2(\B_1,\R^p)\) that satisfy the hypothesis of Theorem~\ref{th_main} with the constant \(K\). We will use extensively the fact that if \(\phi\in\M\), \(\B_{z,r}\subset\B_1\), then the function
\[\zeta\in\B_1\mapsto \phi(z+r\zeta)\]
also belongs to \(\M\) (with the same constant).
\begin{lemma}\label{lem_dev_decrease}
There exist \(\tau\in \left(0,\frac{1}{2}\right]\), \(C_1>0\) such that the following holds: let \(\phi\in\M\), then
\[ D(\phi,\tau )\leq \frac{1}{2} D( \phi,1)+C_1\left(\frac{|\B_{1}\cap\{ \phi=0\}|}{|\B_1|}\right)^\frac{1}{2}.
\]
\end{lemma}
\begin{proof}
Let \(\tau\in\left(0,\frac{1}{2}\right)\) to be fixed. We estimate the deviation of \(\phi\) at the scale \(\tau\):
\begin{align*}
 D(\phi,\tau)&\lesssim  D(\phi_{\B_1},\tau)+ D(\phi-\phi_{\B_1},\tau)&\text{by triangle ineq.}\\
&\lesssim \tau\Vert \nabla \phi_{\B_1}\Vert_{L^\infty(\B_{\tau})}+\tau^{-\frac{d}{2}}\Vert \phi-\phi_{\B_1}\Vert_{L^2(\B_1)}\\
&\lesssim \tau D(\phi_{\B_1},1)+\tau^{-\frac{d}{2}}\Vert \phi-\phi_{\B_1}\Vert_{L^2(\B_1)}&\text{ by eq. }\eqref{eq_hyp2}\\
&\lesssim \tau D(\phi,1)+\tau^{-\frac{d}{2}}\Vert \phi-\phi_{\B_1}\Vert_{L^2(\B_1)}&\text{by triangle ineq.}\\ &\lesssim\tau D(\phi,1)+\tau^{-\frac{d}{2}}|\B_{1}\cap\{\phi=0\}|^\frac{1}{2}&\text{ by eq. }\eqref{eq_hyp1}
\end{align*}
This implies the result for a sufficiently small \(\tau\in\left(0,\frac{1}{2}\right)\), for some \(C_1>0\). Here \(\tau\) and \(C_1\) are chosen depending only on \(d,K\).
\end{proof}

\begin{corollary}\label{cor_BMO}
There exist \(C_2,\alpha>0\) such that for any \(\phi\in\M\), for any \(z\in \B_{\frac{3}{4}}\), \(r\in\left(0,1-|z|\right)\), we have
\[ D( \phi(z+\cdot) , r)\leq C_2\left(1+r^\alpha\Vert \phi\Vert_{L^2(\B_1)}\right).\]
\end{corollary}
\begin{proof}
It is sufficient to prove the result for \(z=0\). Let \(\tau,C_1\) be given by Lemma~\ref{lem_dev_decrease}. Let \(k\in\N\) be chosen such that
\[\tau^{k+1}< r\leq\tau^k\]
i.e. \(k=\left\lfloor \frac{\log(r)}{\log(\tau)}\right\rfloor\). Then iterating the previous lemma we have
\begin{align*}
 D(\phi,r)&\leq \frac{1}{2^k} D(\phi,\tau^{-k}r)+C_1\sum_{\ell=1}^{k}\frac{1}{2^{\ell-1}}\left(\frac{|\B_{\tau^{-\ell}r}\cap\{\phi=0\}|}{|\B_{\tau^{-\ell}r}|}\right)^\frac{1}{2}\\
&\leq \frac{1}{2^k} D(\phi,\tau^{-k}r)+C_1\sum_{\ell=1}^{k}\frac{1}{2^{\ell-1}}\\
&\lesssim r^{\alpha}  D(\phi,1)+1\text{ with }\alpha:=\frac{\log\left(\frac{1}{2}\right)}{\log(\tau)}>0.
\end{align*}
This implies the result, since \( D(\phi,1)\lesssim \Vert\phi\Vert_{L^2(\B_1)}\).
\end{proof}
Corollary~\ref{cor_BMO} is sufficient to obtain that \([\phi]_{\mathrm{BMO}(\B_{\frac{1}{2}})}\lesssim 1+\Vert \phi\Vert_{L^1(\B_1)}\). This is not sufficient to prove that \(\phi\) is bounded since mean oscillation bounds do not imply \(L^\infty\) bounds, and the next lemma is the crucial point to obtain an estimate beyond BMO: when the average of \(\phi\) is sufficiently large compared to its deviation, then the measure of the zero set \(\{ \phi=0\}\) must decrease rapidly at smaller radii, which will in turn improve the oscillation bound at smaller scales.
\begin{lemma}\label{lem_supp_decrease}
There exists \(M_1>0\) such that for any \(\phi\in\M\), if
\[\left|\fint_{\B_1}\phi\right|\geq M_1\left(D(\phi,1)+1\right),\]
then
\[\frac{|\B_\tau\cap\{\phi=0\}|}{|\B_{\tau}|}\leq \frac{1}{4}\frac{|\B_1\cap\{\phi=0\}|}{|\B_1|},\]
where \(\tau\in (0,\frac{1}{2}]\) is defined in Lemma~\ref{lem_dev_decrease}.
\end{lemma}

\begin{proof}
We start from inequality~\eqref{eq_hyp1}:
\[\int_{\B_{1}}|\phi-\phi_{\B_1}|^2\leq |\B_{1}\cap\{ \phi=0\}|.\]
By restricting the integral on the left-hand side to \(\B_\tau\cap\{\phi=0\}\), we obtain
\begin{equation*}\label{eq_aux_1}
|\B_{\tau}\cap\{ \phi=0\}|\left(\inf_{\B_{\tau}}\left|\phi_{\B_1}\right|\right)^2\leq |\B_{1}\cap\{ \phi=0\}|.
\end{equation*}
We estimate the factor on the left-hand side:
\begin{align*}
\inf_{\B_{\tau}}\left|\phi_{\B_1}\right|&\geq \left|\fint_{\B_\frac{1}{2}}\phi_{\B_1}\right|-\left\Vert \nabla \phi_{\B_1}\right\Vert_{L^\infty(\B_{\frac{1}{2}})}&\text{ since }\tau\leq \frac{1}{2}\\
&\geq  \left|\fint_{\B_1}\phi\right|-\fint_{\B_\frac{1}{2}}\left|\phi_{\B_1}-\phi\right|-\fint_{\B_\frac{1}{2}}\left|\phi-\fint_{\B_1}\phi\right|- \left\Vert \nabla \phi_{\B_1}\right\Vert_{L^\infty(\B_{\frac{1}{2}})}&\text{ by triangle ineq. }\\
&\geq\left|\fint_{\B_1}\phi\right|-2^\frac{d}{2}-2^\frac{d}{2}D(\phi,1)-K D(\phi_{\B_1},1)&\text{ by eq. }\eqref{eq_hyp1},~\eqref{eq_hyp2}\\
&\geq\left|\fint_{\B_1}\phi\right|-2^\frac{d}{2}-2^\frac{d}{2}D(\phi,1)-K D(\phi_{\B_1}-\phi,1)-K D(\phi,1)\\
&\geq\left|\fint_{\B_1}\phi\right|-\left(2^\frac{d}{2}+2K\right)-\left(2^\frac{d}{2}+K\right)D(\phi,1)&\text{ by eq. }\eqref{eq_hyp1}.
\end{align*}
This is larger than \(2\tau^{-\frac{d}{2}}\) when the ratio \(\frac{\left|\fint_{\B_1}\phi\right|}{D(\phi,1)+1}\) is sufficiently large, which proves the result.
\end{proof}
The previous lemma may then be iterated as follows.
\begin{lemma}\label{lem_supp_decrease_iterate}
There exists \(M_2>0\) such that for any \(\phi\in\M\), if
\[\left|\fint_{\B_1}\phi\right|\geq M_2\left( D(\phi,1)+1\right),\]
then the set \(\{\phi=0\}\) has Lebesgue density \(0\) at the origin.
\end{lemma}
\begin{proof}
Let \(\tau, C_1\) be defined by Lemma~\ref{lem_dev_decrease}, \(M_1\) defined by Lemma~\ref{lem_supp_decrease}. Let \(M_2>M_1\) to be fixed, assume that
\[\left|\fint_{\B_1}\phi\right|\geq M_2\left( D(\phi,1)+1\right)\]
We prove that if \(M_2\) is sufficiently large, then for any \(k\in\N^*\), we have
\begin{align*}
\left|\fint_{\B_{\tau^k}}\phi\right|&\geq M_1\left( D(\phi,\tau^k)+1\right).
\end{align*}
We proceed by induction. Assume that for some \(k\in\N^*\), we have for any \(\ell=0,1,2,\hdots,k-1\):
\begin{equation}\label{eq_inductionl}
\left|\fint_{\B_{\tau^\ell}}\phi\right|\geq M_1\left( D(\phi,\tau^\ell)+1\right)
\end{equation}
Then for every \(\ell\in \{0,1,\hdots,k-1\}\) we have
\begin{align*}
\frac{|\B_{\tau^{\ell+1}}\cap\{\phi=0\}|}{|\B_{\tau^{\ell+1}}|}&\leq \frac{1}{4}\frac{|\B_{\tau^{\ell}}\cap\{\phi=0\}|}{|\B_{\tau^\ell}|}\\
 D(\phi,\tau^{\ell+1})&\leq \frac{1}{2} D(\phi,\tau^{\ell})+C_1\left(\frac{|\B_{\tau^\ell}\cap\{\phi=0\}|}{|\B_{\tau^\ell}|}\right)^\frac{1}{2}\\
\left|\fint_{\B_{\tau^{\ell+1}}}\phi\right|&\geq \left|\fint_{\B_{\tau^\ell}}\phi\right|-\tau^{-\frac{d}{2}} D(\phi,\tau^{\ell})
\end{align*}
Indeed the first estimate is obtained by Lemma~\ref{lem_supp_decrease}, the second by Lemma~\ref{lem_dev_decrease}, and the third is direct by the triangle and Cauchy--Schwarz inequalities. This implies, for every \(\ell=0,1,\hdots,k\):
\begin{equation}\label{eq_aux_decreasevolume}
|\B_{\tau^{\ell}}\cap\{\phi=0\}|\leq 4^{-\ell}|\B_{\tau^{\ell}}|,
\end{equation}
\begin{equation}\label{eq_aux_deviation}
\begin{split}
 D(\phi,\tau^{\ell})&\leq 2^{-\ell} D(\phi,1)+C_1\sum_{j=0}^{\ell-1}2^{-j}\left(\frac{|\B_{\tau^{\ell-j-1}}\cap\{\phi=0\}|}{|\B_{\tau^{\ell-j-1}}|}\right)^\frac{1}{2}\\
&\leq 2^{-\ell} D(\phi,1)+C_1\sum_{j=0}^{\ell-1}2^{1-\ell}\text{ by eq. }\eqref{eq_aux_decreasevolume}\\
&= 2^{-\ell} D(\phi,1)+C_1 \ell 2^{1-\ell},
\end{split}
\end{equation}
and
\begin{align*}
\left|\fint_{\B_{\tau^k}}\phi\right|&\geq \left|\fint_{\B_{1}}\phi\right|-\tau^{-\frac{d}{2}}\sum_{\ell=0}^{k-1} D(\phi,\tau^\ell)\\
&\geq  \left|\fint_{\B_{1}}\phi\right|-\tau^{-\frac{d}{2}}\sum_{\ell=0}^{k-1}\left(2^{-\ell} D(\phi,1)+C_1 \ell 2^{1-\ell}\right)\text{ by eq. }\eqref{eq_aux_deviation}\\
&\geq M_2\left( D(\phi,1)+1\right)-2\tau^{-\frac{d}{2}} D(\phi,1)-4C_1\tau^{-\frac{d}{2}}
\end{align*}
Combining this with~\eqref{eq_aux_deviation}, which simplifies to \( D(\phi,\tau^k)\leq  D(\phi,1)+C_1\), we obtain
\begin{align*}
\left|\fint_{\B_{\tau^k}}\phi\right|-M_1\left( D(\phi,\tau^k)+1\right)&\geq \left(M_2-2\tau^{-\frac{d}{2}}-M_1\right) D(\phi,1)+\left(M_2-4C_1\tau^{-\frac{d}{2}}-M_1\left(1+C_1\right)\right)
\end{align*}
For a sufficiently large \(M_2>0\) that depends only on \(d,K\), the right-hand side is positive.
Thus, by induction the equation~\eqref{eq_inductionl} holds for every \(k\in\N\), and by equation~\eqref{eq_aux_decreasevolume} we have
\[
\forall k\in\N,\ |\B_{\tau^{k}}\cap\{\phi=0\}|\leq 4^{-k}|\B_{\tau^{k}}|.\]
\end{proof}

We deduce the proof of the main result, Theorem~\ref{th_main}.
\begin{proof}[Proof of Theorem~\ref{th_main}]
Let \(\phi\) that satisfies the hypothesis of Theorem~\ref{th_main}. Let \(\tau\) be defined by Lemma~\ref{lem_dev_decrease}, \(\alpha,C_2\) defined by Corollary~\ref{cor_BMO}, \(M_2\) defined by Lemma~\ref{lem_supp_decrease_iterate}. Let \(z\in\B_{\frac{1}{2}}\), and
\[r=2\mathrm{dist}(z,\mathcal{F}_\phi).\]
Assume \(r\geq \frac{1}{2}\), and that \(\phi\) is not identically zero on \(\B_{z,\frac{1}{2}}\). By~\eqref{eq_hyp1} we have \(\phi=\phi_{\B_{z,\frac{1}{4}}}\), and by equation~\eqref{eq_hyp2} we have
\[\Vert \nabla \phi\Vert_{L^\infty(\B_{z,\frac{1}{8}})}\lesssim  D\left(\phi(z+\cdot),\frac{1}{4}\right)\lesssim \Vert \phi\Vert_{L^2(\B_1)}.\]
In particular this implies
\[|\phi(z)|\leq \left|\fint_{\B_{z,\frac{1}{8}}}\phi\right|+\frac{1}{8}\Vert \nabla \phi\Vert_{L^\infty(\B_{z,\frac{1}{8}})}\lesssim 1+\Vert \phi\Vert_{L^2(\B_1)}\]
which is the result.

Assume now \(r<\frac{1}{2}\).  Let \(\rho\in \left(r,\frac{1}{2}\right)\). Let \(\zeta\in\B_{z,\frac{1}{2}\rho}\), then \(\zeta\in\B_{\frac{3}{4}}\) so by Corollary~\ref{cor_BMO} we have
\begin{align*}
 D(\phi(z+\cdot),\rho)&\leq  C_2\left(1+\rho^\alpha\Vert \phi\Vert_{L^2(\B_1)}\right),\\
 D(\phi(\zeta+\cdot),\frac{1}{2}\rho)&\leq  C_2\left(1+\rho^\alpha\Vert \phi\Vert_{L^2(\B_1)}\right).
\end{align*}
Then
\begin{align*}
\left|\fint_{\B_{\zeta,\frac{1}{2}\rho}}\phi\right|-M_2 \left( D\left(\phi(\zeta+\cdot),\frac{1}{2}\rho\right)+1\right)&\geq \left|\fint_{\B_{z,\rho}}\phi\right|-2^{\frac{d}{2}} D\left(\phi(z+\cdot),\rho\right)-M_2\left( D\left(\phi(\zeta+\cdot),\frac{1}{2}\rho\right)+1\right)\\
&\geq \left|\fint_{\B_{z,\rho}}\phi\right|-C_3\left(1+\rho^\alpha\Vert \phi\Vert_{L^2(\B_1)}\right)
\end{align*}
where \(C_3:=\left(M_2+2^\frac{d}{2}\right)C_2+M_2\). Thus, if \(\left|\fint_{\B_{z,\rho}}\phi\right|\geq C_3\left(1+\rho^\alpha\Vert \phi\Vert_{L^2(\B_1)}\right)\), then by Lemma~\ref{lem_supp_decrease_iterate} applied to \(\phi(\zeta+\frac{1}{2}\rho\cdot)\) for every \(\zeta\in\B_{z,\frac{1}{2}\rho}\), we obtain that \(\{\phi=0\}\) has measure \(0\) in \(\B_{z,\frac{1}{2}\rho}\), which contradicts the fact that \(\mathcal{F}_\phi\cap\B_{z,\frac{1}{2}\rho}\neq \emptyset\) by definition of \(r\).
By this contradiction, we obtain
\begin{equation}\label{eq_aux_control_avg}
\forall \rho \in\left(r,\frac{1}{2}\right),\ \left|\fint_{\B_{z,\rho}}\phi\right|\leq C_3\left(1+\rho^\alpha\Vert \phi\Vert_{L^2(\B_1)}\right)
\end{equation}
To conclude, we differentiate the cases \(r>0\) and \(r=0\):
\begin{itemize}
\item \(r>0\): Since \(\phi=\phi_{\B_{z,\frac{1}{2}r}}\), we have by equation~\eqref{eq_hyp2}:
\begin{align*}
| \phi(z)|&\lesssim \left|\fint_{\B_{z,\frac{1}{4}r}} \phi\right|+r\Vert\nabla \phi\Vert_{L^\infty(\B_{z,\frac{1}{4}r})}\\
&\lesssim \left|\fint_{\B_{z,r}}\phi\right|+ D(\phi(z+\cdot),r)+ D(\phi(z+\cdot),\frac{1}{2}r)\text{ by eq. }\eqref{eq_hyp2}\text{ since }\phi=\phi_{\B_{z,\frac{1}{2}r}}\text{ in }\B_{z,\frac{1}{2}r}\\
&\lesssim 1+r^\alpha \Vert \phi\Vert_{L^2(\B_1)}\text{ by eq. }\eqref{eq_aux_control_avg}\text{ with }\rho\to r\text{ and Corollary }\ref{cor_BMO}
\end{align*}
This concludes the proof.
\item \(r=0\): in this case, taking \(\rho\to 0\) in~\eqref{eq_aux_control_avg}, we obtain (for every \(z\) that is a  Lebesgue point of \(\phi\)): \(|\phi(z)|\leq C_3\).
\end{itemize}
\end{proof}
\section{Applications to minimizers of generalized Bernoulli problems}\label{sec_vectorialBernoulli}
In this section, we explain how the main result, Theorem~\ref{th_main}, applies to general Bernoulli problems, in particular we prove Corollaries~\ref{cor_Q} and~\ref{cor_contrainte}.
\subsection{Vectorial and higher order Bernoulli problems}
We first remind some known results on linear elliptic systems: in this subsection, we fix \(d,n,p\geq 1\), and \(Q\) a quadratic form on \((\R^d)^{\otimes n}\otimes \R^p\) that satisfies the Legendre--Hadamard condition~\eqref{eq_legendrehadamard}.
\begin{lemma}\label{lem_Cacciopoli}
There exists \(C>0\) such that the following holds. Let \(v\in H^n(\B_1,\R^p)\) be a local minimizer of
\[w\in H^n(\B_1,\R^p)\mapsto \int_{\B_1}Q(\nabla^n w).\]
Then
\[\Vert \nabla v\Vert_{L^{\infty}(\B_{\frac{1}{2}})}\leq C\Vert v\Vert_{L^2(\B_1)}.\]
\end{lemma}
The existence and uniqueness of a local minimizer for any boundary data is a consequence of the Legendre--Hadamard condition.
\begin{proof}
This is a consequence of~\cite[Cor. 22]{Barton}, with \(\dot{\mathbf{F}}=0\) and \(\delta=0\).
\end{proof}
We deduce Corollary~\ref{cor_Q}.
\begin{proof}[Proof of Corollary~\ref{cor_Q}]
Without loss of generality we assume that the Legendre--Hadamard condition~\eqref{eq_legendrehadamard} holds with \(\lambda_Q=1\).
We let \(\phi=\nabla^n u\). For any ball \(\B_{z,r}\subset\B_1\), we define
\[\phi_{\B_{z,r}}:=\nabla^n u_{z,r}\]
where \(u_{z,r}\) is by definition the minimizer of
\[v\in u+H^n_0(\B_{z,r})\mapsto \int_{\B_{z,r}}Q(\nabla^nv).\]
We check that Theorem~\ref{th_main} applies to \(\phi\). For any \(\B_{z,r}\subset\B_1\), we have
\begin{align*}
\int_{\B_{z,r}}|\nabla^n (u-u_{z,r})|^2&\leq \int_{\B_{z,r}}Q(\nabla^n(u-u_{z,r}))&\text{ by Gårding ineq.}\\
&=\int_{\B_{z,r}}Q(\nabla^n u)-Q(\nabla^n u_{z,r})&\text{ by minimality of }u_{z,r}\\
&\leq |\B_{z,r}\cap\{\nabla^m u=0\}|&\text{ by minimality of }u\\
&\leq |\B_{z,r}\cap\{\nabla^n u=0\}|&\text{ since }\nabla^nu \underset{a.e.}{=}0\text{ in }\{\nabla^m u=0\}
\end{align*}
which gives the condition~\eqref{eq_hyp1}. By Lemma~\ref{lem_Cacciopoli} applied to each vectorial component of
\[\nabla^nu_{z,r}-\fint_{\B_{z,r}}\nabla^n u_{z,r},\]
we get
\[r\Vert \nabla^{n+1}u_{z,r}\Vert_{L^\infty(\B_{z,\frac{1}{2}r})}\leq K \left(\fint_{\B_{z,r}}\left|\nabla^n u_{z,r}-\fint_{\B_{z,r}}\nabla^n u_{z,r}\right|^2\right)^\frac{1}{2}\]
for some constant \(K>0\) that depends only on \(Q\). This is exactly the condition~\eqref{eq_hyp2}. Thus Theorem~\ref{th_main} applies to \(\nabla^n u\), and we obtain the result.
\end{proof}

We now prove the non-degeneracy: the proof is more classical and relies on a combination of Caccioppoli inequality and higher integrability given by Sobolev embedding, following \cite[Lem. 14]{LN25}.
\begin{proof}[Proof of Proposition~\ref{prop_nondegeneracy}]
We denote \(u_r=r^{-n}u(r\cdot)\), so that \(u_r\) is also a local minimizer of the same functional in \(\B_1\) for every \(r\in (0,1)\). We write \(a\lesssim b\) when \(a\leq Cb\) for some constant \(C>0\) depending only on \(d,p,n,Q\).

Since \(m=0\), applying the Caccioppoli inequality (obtained by comparing \(u_r\) and \(e^{t\eta^b}u_r\) for sufficiently large \(b\in\N \), \(t\to 0\) and \(\eta\in\mathcal{C}^\infty_c(\B_1,\R)\)), we obtain (see for instance~\cite[Cor. 23]{Barton}):
\[\Vert u_r\Vert_{H^n(\B_\frac{1}{2})}\lesssim \Vert u_r\Vert_{L^2(\B_1)}\]
Fix some \(\varphi\in\mathcal{C}^\infty_c(\B_1,\R)\) such that \(\varphi\equiv 1\) on \(\B_{\frac{1}{2}}\). Using \((1-\varphi)u_r\) as a competitor for \(u_r\) we obtain
\begin{align*}
|\B_{\frac{1}{2}}\cap\{u_r\neq 0\}|&\leq \int_{\B_1}\left(Q(\nabla^n[ (1-\varphi)u_r])-Q(\nabla^n u_r)\right)\lesssim \Vert u_r\Vert_{H^n(\B_1)}^2
\end{align*}
Combining these two estimates and the embedding \(H^n(\B_1)\hookrightarrow L^q(\B_1)\) with \(q=\min\left(4,\frac{2d}{d-2}\right)>2\), we obtain for any \(r\in \left(0,1\right)\):
\begin{align*}
\Vert u_{r/4}\Vert_{L^2(\B_1)}&\leq |\{u_{r/4}\neq 0\}\cap\B_1|^{\frac{1}{2}-\frac{1}{q}}\Vert u_{r/4}\Vert_{L^q(\B_1)}\lesssim |\{u_{r/4}\neq 0\}\cap\B_1|^{\frac{1}{2}-\frac{1}{q}}\Vert u_{r/4}\Vert_{H^n(\B_1)}\\
&\lesssim \Vert u_{r/2}\Vert_{H^n(\B_1)}^{2-\frac{2}{q}}\lesssim\Vert u_r\Vert_{L^2(\B_1)}^{2-\frac{2}{q}}
\end{align*}
Since \(2-\frac{2}{q}>1\), there exists \(c=c(d,p,n,Q)>0\) such that
\[\forall r\in (0,1],\ \Vert u_r\Vert_{L^2(\B_1)}\leq c\text{ implies }\Vert u_{r/4}\Vert_{L^2(\B_1)}\leq \frac{1}{2}\Vert u_{r}\Vert_{L^2(\B_1)}.\]
If \(\Vert u\Vert_{L^2(\B_1)}\leq c\), then by induction we obtain \(u(0)=0\), which is the result.
\end{proof}

\begin{remark}\label{rk_nondeg}
When \(m\geq 1\), one could ask if a similar result, of the form
\[\nabla^m u(0)\neq 0\ \Rightarrow\ \Vert u\Vert_{H^n(\B_1)}\geq c\]
holds. However no such non-degeneracy result holds; we refer to Proposition~\ref{prop_pipe} in the case \(d=2\), which corresponds to local minimizers of
\[u\in H^2(\B_1,\R)\mapsto \int_{\B_1}\left(|\nabla^2 u|^2+\chi_{\nabla u\neq 0}\right)\]
through the identification \(\nabla^\bot u=v\). The central issue is the loss of a ``Caccioppoli-type'' inequality that holds independently of the domain.
One may also prove through more elementary methods that the function
\[u_\eps(x)=\begin{cases}-\frac{\eps^2}{3} & \text{ if }x\leq -\eps\\ \frac{\eps x}{2}-\frac{x^3}{6\eps} & \text{ if } -\eps\leq x\leq \eps\\ +\frac{\eps^2}{3} & \text{ if } x\geq \eps\end{cases}\]
is a local minimizer of \(u\in H^2([-1,1],\R)\mapsto \int_{-1}^{1}\left(|u''|^2+\chi_{u'\neq 0}\right)\) for any \(\eps\in (0,1)\), which contradicts the non-degeneracy.
\end{remark}
\subsection{Constrained Bernoulli problems}
The proof follows the same outline as the proof of Corollary~\ref{cor_Q}.
\begin{lemma}\label{lem_Cacciopoli_Stokes}
There exists \(C>0\) such that the following holds. Let \(v\in H^1_{\mathrm{div}}(\B_1,\R^d)\) be a local minimizer of
\[w\in H^1_{\mathrm{div}}(\B_1,\R^d)\mapsto \int_{\B_1}|\nabla w|^2\]
Then
\[\Vert \nabla v\Vert_{L^{\infty}(\B_{\frac{1}{2}})}\leq C\Vert v\Vert_{L^2(\B_1)}.\]
\end{lemma}
\begin{proof}
Assume \(v:\B_1\to\R^d\) satisfies Stokes's equation, then by~\cite[Th. 1]{Kratz} we have for all \(z\in\B_1\):
\[v(z)=\int_{\partial\B_1}S(z,\zeta)v(\zeta)d\zeta\]
where \(S\in\mathcal{C}^\infty(\B_1\times\partial\B_1,\R^{d\times d })\) is defined in~\cite[eq. (5)]{Kratz}. Let \(\eta\in\mathcal{C}^\infty_c(\B_1\setminus\B_{\frac{3}{4}})\) be a radial function such that \(\int_{0}^{1}\eta(r)dr=1\), then denoting
\[S_\eta(z,\zeta):=\frac{\eta(\zeta)}{|\zeta|^{d-1}}S\left(\frac{z}{|\zeta|},\frac{\zeta}{|\zeta|}\right),\]
we have, since \(v(r\cdot)\) satisfies Stokes' equation for every \(r\in \left(\frac{1}{2},1\right)\),
\[\forall z\in\B_{\frac{1}{2}},\ v(z)=\int_{B_1} S_\eta(z,\zeta)v(\zeta)d\zeta.\]
As a consequence we get
\[\forall z\in\B_{\frac{1}{2}},\ \nabla v(z)=\int_{B_1}\nabla_z S_\eta(z,\zeta)v(\zeta)d\zeta.\]
Since \(\nabla_z S_\eta\) is bounded on \(\B_{\frac{1}{2}}\times\B_1\), we obtain the result.
\end{proof}

\begin{proof}[Proof of Corollary~\ref{cor_contrainte}]
The proof is identical to the proof of Corollary~\ref{cor_Q}: we let \(\phi=\nabla u\) and \(\phi_{\B_{z,r}}=\nabla u_{z,r}\) where \(u_{z,r}\) is the solution of Stokes's equation in \(\B_{z,r}\) with boundary data \(u\), and condition~\eqref{eq_hyp1} is obtained by the minimality of \(u\) and \(u_{z,r}\), while the condition~\eqref{eq_hyp2} is obtained as a consequence of Lemma~\ref{lem_Cacciopoli_Stokes} applied to each component of \(\nabla u_{z,r}-\fint_{\B_{z,r}}\nabla u_{z,r}\).
\end{proof}

We may give another application on the minimization of \(\mathcal{D}(K)\) defined in equation~\eqref{eq_drag}, among all closed sets of given volume in \(\R^d\) (\(d\geq 3\)), and \(u_\infty\in\R^d\setminus\{0\}\). Indeed this optimization problem admits the natural relaxation
\begin{equation}\label{eq_relaxed_drag}
\inf\left\{\left|\{u=0\}\right|^\frac{2-d}{d}\int_{\R^d}|\nabla u|^2,\ u\in u_\infty+\dot{H}^1_{\mathrm{div}}(\R^d,\R^d)\right\}.
\end{equation}
\begin{proposition}\label{prop_drag}
Let \(u\) be a minimizer of~\eqref{eq_relaxed_drag}, then \(\nabla u\in L^{\infty}(\R^d,\R^d)\).
\end{proposition}
In particular, this implies that the set \(K=\{u=0\}\) is a closed set that achieves the minimal value of \(|K|^{\frac{2-d}{d}}\mathcal{D}(K)\). Whether $K$ is bounded is unclear.
\begin{proof}
Without loss of generality, assume \(|\{u=0\}|=1\) and write \(\mathcal{D}=\int_{\R^d}|\nabla u|^2\). Let \(\ov{r}>0\) be such that \(|\B_{\ov{r}}|=1\), then for any ball \(B\subset\R^d\) of radius \(r\in (0,\ov{r})\), if we denote \(u_{B}\) the solution of Stokes' equation in \(B\) with boundary data \(u|_{\partial B}\), then by minimality of \(u\) and \(u_B\), we have \(u=u_{B}\) if \(B\cap\mathcal{F}_u=\emptyset\), and otherwise
\begin{align*}
\mathcal{D}&\leq \left(1-|B\cap\{u=0\}|\right)^\frac{2-d}{d}\left(\mathcal{D}-\int_{B}|\nabla (u-u_B)|^2\right).
\end{align*}
This implies
\begin{align*}
\mathcal{D}-\int_{B}|\nabla (u-u_B)|^2\geq \mathcal{D}\left(1-|B\cap\{u=0\}|\right)^\frac{d-2}{d}\geq \mathcal{D}\left(1-|B\cap\{u=0\}|\right).
\end{align*}
This simplifies into \(\int_{B}|\nabla (u-u_{B})|^2\leq \mathcal{D}|B\cap\{u=0\}|\). As previously, by application of Theorem~\ref{th_main} to the family of functions \((\phi,(\phi_B)_B)=\left(\mathcal{D}^{-\frac{1}{2}}\nabla u,\left(\mathcal{D}^{-\frac{1}{2}}\nabla u_{B}\right)_B\right)\), we obtain \(\nabla u\in L^{\infty}(\R^d,\R^d)\).
\end{proof}

No non-degeneracy result similar to Proposition~\ref{prop_nondegeneracy} holds for this problem, the central issue in the proof being the absence of a Caccioppoli inequality
\[\int_{\B_{\frac{1}{2}}}|\nabla u|^2\lesssim \int_{\B_1}|u|^2\]
for a function \(u\in H^1(\B_1,\R^d)\) that satisfies Stokes' equation on its support. A counterexample to this Caccioppoli inequality is given by
\[u_\eps(x,y)=\left(0,\frac{\eps^2-x^2}{2\eps}\chi_{|x|<\eps}\right)\]
as \(\eps\to 0\), associated to the pressure \(p_\eps(x,y)=-\frac{y}{\eps}\chi_{|x|<\eps}\). This corresponds to a ``pipe'' with a standard parabolic velocity profile. The same example gives in fact a counterexample to the non-degeneracy, in any dimension. We follow closely the argument of~\cite[Th. 4.4]{Schulz13}, which treats the three-dimensional, volume-constrained case.

Below we denote by $\B_r^{d-1}$ the centered open ball of radius $r$ in $\R^{d-1}$.
\begin{proposition}\label{prop_pipe}
Let $d\geq 2$, $\Om=\B_{1}^{d-1}\times (-1,1)$, $\eps\in (0,1)$ and
\[v_\eps(x,y)=\left(0,0,\hdots,0,\frac{\eps^2-\Vert x\Vert^2}{2\eps}\chi_{\Vert x\Vert<\eps}\right)\]
Then $v_\eps$ is a local minimizer of
\[v\in H^1_{\mathrm{div}}(\Om,\R^d)\mapsto \int_{\Om}\left(|\nabla v|^2+\chi_{v\neq 0}\right).\]
\end{proposition}
In particular $\Vert v_\eps\Vert_{H^1(\Om)}+|\{v_\eps\neq 0\}|$ converges to $0$ as $\eps\to 0$, yet $v_\eps$ does  not vanish in a neighborhood of the origin.
\begin{proof}
Consider $v\in v_\eps+H^1_{0,\mathrm{div}}(\Om)$, we write $v=(v^x,v^y)$ where $v^x\in H^1(\Om,\R^{d-1})$ and $v^{y}\in H^1(\Om,\R)$. For any $y\in [-1,1]$, write
\[\Om^y=\B_1^{d-1}\times \{y\}.\]
Moreover, we write $c_d$ the measure of the  $(d-1)$-dimensional unit ball and $R(y)>0$ such that
\[\mathcal{H}^{d-1}(\Om^y\cap\{v\neq 0\})=c_d R(y)^{d-1}.\]
Since  $\mathrm{div}(v)=0$ and $\int_{\Om^{1}}v^y=\frac{c_d }{d+1}\eps^{d}$, then for all $y\in (-1,1)$ we have
\[\int_{\Om^y}v^y=\frac{c_d }{d+1}\eps^{d},\]
By a radial rearrangement argument (see \cite[Lem. 4.6]{Schulz13} for the two-dimensional case), we have for almost every $y\in (-1,1)$:
\begin{equation}\label{eq_talenti}
\int_{\Om^y}|\nabla_x v^y|^2\geq c_d\frac{d-1}{d+1}\frac{\eps^{2d}}{R(y)^{d+1}}.
\end{equation}
Indeed, denoting $v^y_*(\cdot,y)$ the radially decreasing Schwarz rearrangement of $|v^y|(\cdot,y)$, we have
\[\int_{\Om^y}|\nabla_x v^y_*|^2\leq \int_{\Om^y}|\nabla_x v^y|^2,\ \int_{\Om^y}v^y_*\geq \int_{\Om^y}v^y=\frac{c_d }{d+1}\eps^{d}\]
and the minimal value of $\int_{\B_{R(y)}^{d-1}}|\nabla_x w|^2$ under a constraint on $\int_{\B_{R(y)}^{d-1}}w$, for $w\in H^1_0(\B_{R(y)}^{d-1})$, is given by a multiple of $(R(y)^2-\Vert x\Vert^2)_+$, which gives the estimate \eqref{eq_talenti}.
Then
\begin{align*}
\int_{\Om}\left(|\nabla v|^2+\chi_{v\neq 0}\right)&\geq c_d\int_{-1}^{1}\left(\frac{d-1}{d+1}\frac{\eps^{2d}}{R(y)^{d+1}}+ R(y)^{d-1}\right)dy\\
&\geq \frac{2dc_d}{d+1}\left(\int_{-1}^{1}\frac{\eps^{2d}}{R(y)^{d+1}}dy\right)^\frac{d-1}{2d}\left(\int_{-1}^{1}R(y)^{d-1}dy\right)^\frac{d+1}{2d}&\text{ by Young ineq.}\\
&\geq\frac{2dc_d}{d+1}\int_{-1}^{1}\left(\frac{\eps^{2d}}{R(y)^{d+1}}\right)^\frac{d-1}{2d}\left(R(y)^{d-1}\right)^\frac{d+1}{2d}dy&\text{ by Hölder ineq.}\\
&=\frac{4dc_d}{d+1}\eps^{d-1}\\
&=\int_{\Om}\left(|\nabla v_\eps|^2+\chi_{v_\eps\neq 0}\right).
\end{align*}
\end{proof}

\section{Lipschitz regularity of stationary solutions}\label{sec_proof_main_crit}
In  this section  we give the proof of Corollary~\ref{cor_lipschitz_crit}  and of Theorem~\ref{th_crit_vectorial_intro} on non-minimizing solutions of the Bernoulli problem. We will assume \(d=2\) throughout the section.

The starting point is to reformulate the stationarity condition in terms of the Hopf differential. Recall that for a function \(u\in H^{1}_{\mathrm{loc}}(\R^2,\mathbb{R}^p)\), the Hopf differential of \(u\) is defined as
\[
 \Ho{u}=\partial_{z}u\cdot \partial_{z}u=\sum_{j=1}^p (\partial_z u_j)^2=\frac{1}{4}\bigl(|\partial_x u|^2-|\partial_y u|^2-2 \iu \partial_x u \cdot\partial_y u \bigr).
\]
Here \(\iu\) is the imaginary unit, we are identifying \(\mathbb{R}^2\) with \(\mathbb{C}\) and  we have set \(\partial_z =\frac{1}{2}\left(\partial_x-\iu \partial_y\right)\) and \(\partial_{\ov{z}} =\frac{1}{2}\left(\partial_x+\iu \partial_y\right)\).

\subsection{\(L^\infty\) bound on the Hopf differential and proof of Corollary~\ref{cor_lipschitz_crit}}
\begin{lemma}\label{lem_equ}
 For any open set \(D\subset\R^2\), \(u\in H^1(D,\mathbb{R}^p)\) is a stationary solution of the Bernoulli problem if and only if the equation
 \[
  \partial_{\ov{z}}\Ho{u}+\frac{1}{4}\partial_{z}\chi_{u=0}=0
 \]
holds in \(D\), in the weak sense.
\end{lemma}
% We remind that with only this condition, even in the support of \(u\), \(u\) may not be a harmonic function. However in this case we have \(\partial_{z}(\partial_{\ov{z}}u)^2=0\), so \((\partial_{\ov{z}}u)^2\) is harmonic, meaning it still verifies good regularity estimates.
\begin{proof}
The stationarity condition is equivalent to
\[\forall\xi\in\mathcal{C}^\infty_c(D,\R^2),\ \int_{D}\left(|\nabla u|^2\mathrm{div} \xi -2 \sum_{j=1}^p D\xi\nabla u_j\cdot\nabla u_j-\chi_{u=0}\mathrm{div} \xi\right)=0\]
Taking \(\xi\) with only \(x\) component, and \(\xi\) with only \(y\) component, we obtain the two conditions \begin{align*}
\partial_x\left(-|\partial_x u|^2+|\partial_y u|^2-\chi_{u=0}\right)-2\partial_y\left(\partial_x u\cdot  \partial_y u\right)&=0\\
\partial_y\left(+|\partial_x u|^2-|\partial_y u|^2-\chi_{u=0}\right)-2\partial_x\left(\partial_x u\cdot \partial_y u\right)&=0
\end{align*}
Multiplying the second equation by \(\iu\) and  subtracting them gives the desired conclusion.
% We identify \[(2\partial_{\ov{z}}u)^2=\left(\partial_x u+i\partial_y u\right)^2= (\partial_x u)^2-(\partial_y u)^2+2i\partial_x u\partial_y u\]
% and so
% \begin{align*}
% \partial_z(2\partial_{\ov{z}} u)^2&=\frac{1}{2}\partial_x\left( (\partial_x u)^2-(\partial_y u)^2\right)+\partial_y\left(\partial_x u \partial_y u\right)\\
% &-i\frac{1}{2}\partial_y\left((\partial_x u)^2-(\partial_y u)^2\right)+i\partial_x\left(\partial_x u \partial_y u\right)\\
% &=-\partial_{\ov{z}} \chi_{u=0}
% \end{align*}
\end{proof}

We now  recall the definition of the Ahlfors--Beurling transform:

\begin{definition}\label{def_B}
Let \(\mathcal{B}:L^2(\R^2)\to L^2(\R^2)\) be  defined by
\[\widehat{\mathcal{B}[f]}(\xi,\zeta)=\frac{\xi-\iu \zeta}{\xi+\iu \zeta}\widehat{f}(\xi,\zeta)\]
where \(\widehat{f}\) is the Fourier transform of \(f\).
\end{definition}

For any measurable function \(h:\Om\to \C\) for \(\Om\subset\R^2\), we set
\[
 \Vert h\Vert_{L^{1,\infty}(\Om
   )}=\sup_{t>0}\Big(t\left|\Om\cap\{|h|>t\}\right|\Big).
\]
The following  are classical mapping properties of \(\mathcal{B}\),~\cite{AIM09}.

\begin{lemma}\label{lem_B}
\(\Vert \mathcal{B}\Vert_{L^2(\R^2)\to L^2(\R^2)}=1\) and \(\Vert \mathcal{B}\Vert_{L^{1}(\R^2)\to L^{1,\infty}(\R^2)}\) is finite. Moreover, for any \(f\in L^2(\R^2)\), we have \(\partial_{\ov{z}}\mathcal{B}[f]=\partial_{z}f\).
\end{lemma}
\begin{proof}
By the Plancherel theorem, \(\mathcal{B}\) is an isometry from \(L^2(\R^2)\) to \(L^2(\R^2)\). \(\mathcal{B}\) is bounded from \(L^1(\R^2)\) to \(L^{1,\infty}(\R^2)\) by the Hörmander--Mihlin theorem (see~\cite{H60} or the more recent reference~\cite[Th. 6.2.7]{G2008}, or~\cite[Th. 4.5.2]{AIM09} where a bound \(\Vert \mathcal{B}\Vert_{L^1(\R^2)\to L^{1,\infty}(\R^2)}\leq 30\) is given). The second relation is direct by taking the Fourier transform on both sides.
\end{proof}
Eventually we record the following elliptic estimate.

\begin{lemma}\label{lem_L1infty}
Let \(h:\B_1\to \C\) be a harmonic function such that \(\Vert h\Vert_{L^{1,\infty}(\B_1)}\) is finite, then \[\Vert h\Vert_{L^\infty(\B_{\frac{1}{2}})}\lesssim \Vert h\Vert_{L^{1,\infty}(\B_{1})}.\]
\end{lemma}
\begin{proof}
For any \(r\in (0,1)\), we have
\begin{align*}
\Vert h\Vert_{L^2(\B_r)}^2&=\int_{0}^{\Vert h\Vert_{L^\infty(\B_r)}}2t|\B_r\cap\{|h|>t\}|dt\\
&\leq 2\Vert h\Vert_{L^\infty(\B_r)}\Vert h\Vert_{L^{1,\infty}(\B_1)}
\end{align*}
For any \(s,r\) such that \(\frac{1}{2}<s<r<1\), by the mean value formula applied to \(h\) we have
\[\Vert h\Vert_{L^\infty(\B_s)}\leq \frac{\pi^{-\frac{1}{2}}}{r-s}\Vert h\Vert_{L^2(\B_r)}.\]
Let \(r_k=\frac{3}{4}-\frac{1}{2^{k+2}}\), so that
\[\frac{1}{2}=r_0<r_1<r_2<\hdots\to \frac{3}{4}.\]
Then for every \(k\in\N\) we have, combining the two previous inequalities:
\[\Vert h\Vert_{L^\infty(\B_{r_k})}\leq \pi^{-\frac{1}{2}}2^{k+\frac{7}{2}}\Vert h\Vert_{L^{\infty}(\B_{r_{k+1}})}^\frac{1}{2}\Vert h\Vert_{L^{1,\infty}(\B_{1})}^\frac{1}{2}.\]
By induction on \(k\):
\[\Vert h\Vert_{L^\infty(\B_{\frac{1}{2}})}\leq \prod_{\ell=0}^{k-1}\left(\pi^{-\frac{1}{2}}2^{\ell+\frac{7}{2}}\right)^\frac{1}{2^{\ell}}\Vert h\Vert_{L^\infty(\B_{r_k})}^{\frac{1}{2^k}}\Vert h\Vert_{L^{1,\infty}(\B_{1})}^{1-\frac{1}{2^k}}.\]
Taking \(k\to +\infty\), we obtain
\[\Vert h\Vert_{L^\infty(\B_{\frac{1}{2}})}\leq \prod_{\ell=0}^{+\infty}\left(\pi^{-\frac{1}{2}}2^{\ell+\frac{7}{2}}\right)^\frac{1}{2^{\ell}}\Vert h\Vert_{L^{1,\infty}(\B_{1})}.\]
The factor on the right-hand side is finite, so this proves the result.
\end{proof}

We now show that for stationary solutions, \(\Ho{u}\) is always bounded, and universally bounded next to a free boundary point.

\begin{lemma}\label{lem_hopf_bounded}
Let \(d=2\). There exist \(C,\alpha>0\) such that the following holds. Let \(u\in
H^1(\B_1,\R^p)\) be a stationary solution of the Bernoulli problem,~\eqref{eq:critical_vectorial}. Then
\[\Vert \Ho{u}\Vert_{L^\infty(\B_\frac{1}{2})}\leq C\left(1+\Vert \Ho{u}\Vert_{L^1(\B_1)}\right)\]
and more precisely, if \(\mathcal{F}_u\cap\B_1\neq\emptyset\), then for almost every \(z\in \B_\frac{1}{2}\),
\[|\Ho{u}(z)|\leq C\left(1+\mathrm{dist}(z,\mathcal{F}_u)^\alpha
\Vert\Ho{u}\Vert_{L^1(\B_1)}\right).\]
\end{lemma}

Before proving the above lemma, note that it immediately implies  Corollary~\ref{cor_lipschitz_crit} since if \(p=1\) then
\[
|\Ho{u}|=\frac{1}{4}|\nabla u|^2.
\]

\begin{proof}[Proof of Lemma~\ref{lem_hopf_bounded}]
If $\mathcal{F}_u\cap\B_1=\emptyset$ then $\Ho{u}$ is harmonic and there is nothing to prove, so we assume $\mathcal{F}_u\cap\B_1\neq\emptyset$.
Let
\[
\Phi=4\Ho{u}
\]
and for any ball \(\B_{z,r}\subset\B_1\), let
\[
 \Phi_{\B_{z,r}}:=4\Ho{u}+\mathcal{B}\left[\chi_{\B_{z,r}\cap \{u=0\}}\right].\]
By Lemmas~\ref{lem_equ} and~\ref{lem_B}, we have \(\partial_{\ov{z}}\Phi_{\B_{z,r}}=0\) in \(\B_{z,r}\): thus \(\Phi_{\B_{z,r}}\) is harmonic in \(\B_{z,r}\) and satisfies the hypothesis~\eqref{eq_hyp2} for some universal constant \(K>0\). Then, we check that the hypothesis~\eqref{eq_hyp1} holds (up to some scalar factor):
\begin{align*}
\int_{\B_{z,r}}\left|\Phi-\Phi_{\B_{z,r}}\right|^2&=\int_{\B_{z,r}}\left|\mathcal{B}\left[\chi_{\B_{z,r}\cap \{u=0\}}\right]\right|^2\leq\int_{\B_{z,r}}\left|\chi_{\B_{z,r}\cap \{u=0\}}\right|^2&\text{ by lemma }\ref{lem_B}\\
&\leq|\B_{z,r}\cap\{u=0\}|\leq |\B_{z,r}\cap\{\Phi=0\}|.
\end{align*}
Thus, by application of Theorem~\ref{th_main} on every ball \(\B_{z,\frac{1}{4}}\) for \(z\in\B_{\frac{1}{2}}\), for some \(\alpha>0\) we have for almost every \(z\in\B_\frac{1}{2}\):
\[\left|\Ho{u}(z)\right|\lesssim 1+\mathrm{dist}(z,\mathcal{F}_u)^\alpha\Vert \Ho{u}\Vert_{L^2(\B_\frac{3}{4})}.\]
Hence to conclude it is sufficient to prove that the last term is controlled by \(1+\Vert\Ho{u}\Vert_{L^1(\B_1)}\). Indeed, using again the properties of the operator \(\mathcal{B}\), we have
\begin{align*}
4\Vert \Ho{u}\Vert_{L^2(\B_\frac{3}{4})}&\leq \left\Vert \mathcal{B}\left[\chi_{\B_{1}\cap \{u=0\}}\right]\right\Vert_{L^2(\B_\frac{3}{4})}+\Vert \Phi_{\B_1}\Vert_{L^2(\B_{\frac{3}{4}})}\\
&\lesssim \left\Vert \chi_{\B_{1}\cap \{u=0\}}\right\Vert_{L^2(\B_1)}+\Vert\Phi_{\B_1}\Vert_{L^{1,\infty}(\B_1)}&\text{ by Lem. }\ref{lem_B}\text{ and }\ref{lem_L1infty}\\
&\lesssim 1+\left\Vert \mathcal{B}\left[\chi_{\B_{1}\cap \{u=0\}}\right]\right\Vert_{L^{1,\infty}(\B_1)}+\Vert \Ho{u}\Vert_{L^{1,\infty}(\B_1)}\\
&\lesssim 1+\Vert \chi_{\B_{1}\cap \{u=0\}}\Vert_{L^{1}(\B_1)}+\Vert \Ho{u}\Vert_{L^{1}(\B_1)} &\text{ by Lem. }\ref{lem_B}.
\end{align*}

\end{proof}

\begin{remark}\label{rem_counterexample_crit}
As mentioned in the introduction, Corollary~\ref{cor_lipschitz_crit} does not extend to a vectorial setting (with \(u=(u_1,u_2,\hdots,u_p)\)). Indeed if  \(p=2\) one can consider
\[(u_1(z),u_2(z))=(v(z),v(-z)),\]
where
\[v(re^{i\theta})=r^\frac{1}{2}\cos\left(\frac{\theta}{2}\right),\ \forall r>0,\theta\in [-\pi,\pi].\]
\(v\)  itself is not a stationary solution of the Bernoulli problem: denoting by \(\sqrt{\cdot}\) the principal determination of the square root on \(\C\setminus\R_-\) (such that $v(z)=\Re\sqrt{z}$), we have
\[\partial_z v=\frac{1}{4\sqrt{z}},\]
so \((\partial_{z}v)^2=\frac{1}{16z}\): this is an integrable function, which is not harmonic in \(\C\) due to the singularity at the origin, so $\partial_{\ov{z}}\Ho{v}\neq 0(=\partial_z\chi_{v\neq 0})$. However we still have
\[\Ho{u}=(\partial_{z}u_1)^2+(\partial_{z}u_2)^2=\frac{1}{16z}+\frac{1}{16(-z)}=0\]
so \(u\) is a  stationary solution, and is not Lipschitz.
\end{remark}

\subsection{Lipschitz regularity of vector-valued stationary solutions}
The goal of this section is to prove Theorem~\ref{th_crit_vectorial_intro}. In view of Lemma~\ref{lem_hopf_bounded}, it is a consequence of the following:

\begin{theorem}\label{th_crit_vectorial}
There exists \(C>0\) such that the following  holds. Let \(u\in H^1(\B_1,\R^p)\) that satisfies~\eqref{eq_structure_weak}. Then
\[\Vert \nabla u\Vert_{L^\infty(\B_\frac{1}{2})}\leq  C\left(\Vert \Ho{u}\Vert_{L^\infty(\B_1)}^\frac{1}{2}+\Vert u\Vert_{L^2(\B_1)}\right).\]
\end{theorem}
%The \(L^2(\B_1)\) norm on \(\Ho{u}\) may be replaced by any \(L^q(\B_1)\) norm for \(q>1\), and this only intervenes in the proof of lemma \ref{lem_equicontinuity}.

The plan of the proof is as follows.
\begin{itemize}
\item First, we prove (in Lemma~\ref{lem_equicontinuity}) an interior Hölder estimate on \(u\).  The central argument is the observation that
\[\Delta (|u|^2-2(u\cdot e)^2)\geq -8|\Ho{u}|\]
for any unit vector \(e\in\mathbb{S}^{p-1}\), which allows us to make a dichotomy between solutions: either \(u\) satisfies ``good'' regularity estimates in \(\B_\frac{1}{2}\), or \(\int_{\B_1\setminus\B_\frac{1}{2}}|u|^2\lesssim \int_{\B_1\setminus\B_\frac{1}{2}}|\nabla u|^2\), which allows us to prove a polynomial growth for the integral of \(|\nabla u|^2\).
\item Using this equicontinuity, we prove (in Lemma~\ref{lem_compactness}) that any sufficiently large solution must be close to some harmonic conformal map. We deduce a mild decay estimate by compactness ( see Lemma~\ref{lem_firstdecay}). Moreover a harmonic conformal map that vanishes at the origin either satisfies some strong decay estimate (see Lemma~\ref{lem_seconddecay}) or is arbitrarily close to some linear conformal map.
\item Since non-zero linear conformal maps do not vanish outside any neighborhood of the origin, any solution that is sufficiently close to such a map satisfies the equation \(\Delta u=0\) on some large ring domain. We conclude by proving (in Lemma~\ref{lem_analysisunstablemodes}) the absence of singular harmonic modes (i.e. modes of the form \(r^{-|n|}e^{i n \theta}\) and \(\log(r)\)) in the Laurent decomposition of \(u\), using the fact that these modes do not satisfy the mild decay estimate (Lemma~\ref{lem_firstdecay}) previously obtained.
\end{itemize}

\begin{lemma}\label{lem_equicontinuity}
There exists \(\alpha\in (0,1)\) such that the following holds. Let \(u\in H^1(\B_1,\R^p)\) that satisfies~\eqref{eq_structure_weak}, then
\[\Vert u\Vert_{\mathcal{C}^\alpha(\B_\frac{1}{2})}\lesssim \Vert\Ho{u}\Vert_{L^\infty(\B_1)}^\frac{1}{2}+\Norm{u}.\]
\end{lemma}
When \(p=2\), one may notice that \(|\Delta( (u_1+iu_2)^2)|=8|\Ho{u}|\in L^\infty(\B_1)\), which gives a short proof of the Hölder continuity. Instead we prove directly the general case, for which there is generally no degree \(2\) polynomial \(Q\) such that \(|\Delta (Q\circ u)|\) is controlled by the Hopf differential of \(u\).
\begin{proof}
Assume (by homogeneity) that \(\Vert\Ho{u}\Vert_{L^\infty(\B_1)}\leq 1\), it is sufficient to prove the existence of some \(\eps>0\) such that
\begin{equation}\label{eq_growthgradient}
\forall r\in\left(0,\frac{1}{2}\right),\ \int_{\B_r}|\nabla u|^2\lesssim r^\eps \left(1+\int_{\B_1}|u|^2\right).
\end{equation}
We first prove the following trichotomy: there exists some constant \(C\lesssim 1\) such that for any \(r\in \left(0,\frac{1}{2}\right)\), one of the three following statements holds:
\begin{equation}\label{eq_trichotomie1}
\begin{split}
\text{either }&\int_{\B_{2r}\setminus\B_{r}}|u|^2\leq Cr^2\int_{\B_{2r}\setminus\B_{r}}|\nabla u|^2\\
\text{or }&\Vert u\Vert_{L^\infty(\B_r)}\leq Cr\\
\text{or }&\Delta u= 0\text{ in }\B_r.
\end{split}
\end{equation}
To this end, we consider the quantity
\[A:=\frac{\int_{\B_{2r}\setminus\B_{r}}|u|^2}{r^2\int_{\B_{2r}\setminus\B_r}|\nabla u|^2}\]
and prove that if \(A\) is sufficiently large, then either the second or the third condition holds.

There exists some \(\rho\in \left(r,2r\right)\) such that
\[\int_{\partial\B_\rho}|u|^2\geq Ar^2\int_{\partial\B_\rho}|\nabla u|^2.\]
Let \(m=\fint_{\partial\B_\rho}u\). Up to replacing \(u\) with \(Ru\) for an appropriate rotation \(R\), we assume that \(m=|m| e_1\). We have
\[\Vert u-m\Vert_{L^\infty(\partial\B_\rho)}\lesssim r^\frac{1}{2}\Vert \nabla u\Vert_{L^2(\partial\B_\rho)}\lesssim A^{-\frac{1}{2}}r^{-\frac{1}{2}}\Vert u\Vert_{L^2(\partial\B_\rho)}\lesssim A^{-\frac{1}{2}}\left(|m|+\Vert u-m\Vert_{L^\infty(\partial\B_\rho)}\right)\]
so if \(A\) is sufficiently large (which we assume), we have
\[\Vert u-m\Vert_{L^\infty(\partial\B_\rho)}\lesssim A^{-\frac{1}{2}}|m|.\]
Let \(Q(v):=|v|^2-2v_1^2\). For a sufficiently large \(A\), we have
\[\sup_{\partial\B_\rho}Q\circ u\lesssim -|m|^2.\]
Then, using the hypothesis~\eqref{eq_structure} (which implies that \(\Delta\left(u_k^2\right)=2|\nabla u_k|^2\)), we compute
\begin{align*}
\Delta(Q\circ u)&=-\Delta\left(u_1^2\right)+\sum_{k=2}^{p}\Delta\left(u_k^2\right)\\
&=2\left(-|\nabla u_1|^2+\sum_{k=2}^{p}|\nabla u_k|^2\right)\\
&=8\left(-\left|\Ho{u}-\sum_{k=2}^{p}(\partial_z u_k)^2\right|+\sum_{k=2}^{p}|(\partial_z u_k)^2|\right)\\
&\geq -8|\Ho{u}|\text{ by triangle inequality.}
\end{align*}
Since \(\Vert\Ho{u}\Vert_{L^\infty(\B_1)}\leq 1\) we obtain by the maximum principle:
\begin{align*}
\sup_{\B_r}Q\circ u&\leq \sup_{\partial\B_\rho}Q\circ u+C r^2
\end{align*}
for some  \(C\lesssim 1\). Since \(\sup_{\partial\B_\rho}Q\circ u\lesssim -|m|^2\), this leaves two possibilities: either \(|m|\lesssim r\), or \(\sup_{\B_r}Q\circ u<0\). In the first case, we have by subharmonicity of \(|u|^2\):
\[\sup_{\B_r}|u|^2\leq\sup_{\partial\B_\rho}|u|^2\lesssim |m|^2\lesssim r^2.\]
In the second case, we obtain \(\Delta u=0\) in \(\B_r\) by the condition~\eqref{eq_structure_weak}. Let us detail this point: the condition~\eqref{eq_structure_weak} still holds for \(\varphi\in H^1_0(\B_r,\R)\cap L^\infty(\B_r,\R)\) by approximation. Since \(\sum_{j=1}^{p}u_j^2\) is bounded from below in \(\B_r\), then for any \(\psi\in\mathcal{C}^\infty_c(\B_r,\R)\) the function
\[\varphi_k:=\frac{u_k}{\sum_{j=1}^{p}u_j^2}\psi\]
belongs to \(H^1_0(\B_r,\R)\cap L^\infty(\B_r,\R)\). Testing~\eqref{eq_structure_weak} with \(\varphi_k\) and summing with respect to \(k\) gives
\[\forall\psi\in\mathcal{C}^\infty_c(\B_r,\R),\ \forall l=1,\hdots,p,\ \int_{\B_r}\nabla\psi\cdot\nabla u_l=0\]
so \(u\) is harmonic in \(\B_r\). This proves the claim that one of the three options~\eqref{eq_trichotomie1} always holds.\bigbreak

For any \(k=1,\hdots,p\), we have \(\Delta (u_k^2)=2|\nabla u_k|^2\): this relation implies a Caccioppoli inequality
\begin{equation}\label{eq_cacciopoli_criticalpoint}
\forall s\in\left(0,\frac{1}{2}\right],\ \int_{\B_s}|\nabla u|^2\lesssim s^{-2}\int_{\B_{2s}\setminus\B_s}|u|^2.
\end{equation}
So, for any \(r\in\left(0,\frac{1}{2}\right)\), one of the following three relations holds:
\begin{equation*}\label{eq_trichotomie2}
\begin{split}
\text{either }&\int_{\B_{r}}|\nabla u|^2\leq  2^{-\eps}\int_{\B_{2r}}|\nabla u|^2\text{ for some universal }\eps\in (0,1)\\
\text{or }&\int_{\B_{\frac{r}{2}}}|\nabla u|^2\lesssim r^2\\
\text{or }&\sup_{\B_\frac{r}{2}}|\nabla u|^2\lesssim r^{-2}\int_{\B_r}|\nabla u|^2
\end{split}
\end{equation*}
Indeed these correspond to the three possibilities of~\eqref{eq_trichotomie1}, by application of the  Caccioppoli inequality~\eqref{eq_cacciopoli_criticalpoint} to the first two and classical elliptic regularity for harmonic functions on the third one.

These three possibilities imply that for all \(r\in (0,\frac{1}{2})\), we have
\[\int_{\B_r}|\nabla u|^2\lesssim r^{\eps}\left(1+\int_{\B_\frac{1}{2}}|\nabla u|^2\right).\]
By Caccioppoli inequality~\eqref{eq_cacciopoli_criticalpoint} on the right-hand side we obtain the relation~\eqref{eq_growthgradient}, which concludes the proof.
\end{proof}
In what follows we define
\begin{equation*}\label{eq_defS0}
\mathcal{S}=\left\{u\in H^1(\B_1,\R^p):u_k\Delta u_l=0,\ \forall k,l,\ \Vert\Ho{u}\Vert_{L^\infty(\B_1)}\leq 1\text{ and }u(0)=0\right\}.
\end{equation*}
Next we prove that any sufficiently large element of $\mathcal{S}$ must be close to some harmonic conformal map, using the interior Hölder bound.
\begin{lemma}\label{lem_compactness}
Let \((u^{(k)})_{k\in\N}\) be a sequence in \(\mathcal{S}\) such that  \(\Norm{u^{(k)}}\to+\infty\). Let
\[v^{(k)}=\Norm{u^{(k)}}^{-1}u^{(k)},\]
then there exists a subsequence  \((k_i)_{i\in\N}\) such that \(v^{(k_i)}\) converges weakly in \(L^2(\B_1)\) and strongly in \(\mathcal{C}^0_\loc(\B_1)\) to some function \(v\) that is harmonic and conformal, such that \(v(0)=0\) and \(\Norm{v}\leq 1\).
\end{lemma}
Here conformal is taken in the sense \(\Ho{v}=0\): this allows critical points. Note that any such function may be written as
\[v(z)=\sum_{n\geq 1}\Re(v_n z^n)\]
for some sequence of complex coefficients \((v_n)_n\) in \(\C^p\) (which are constrained by the conformality condition), and we have
\[\Norm{v}^2=\frac{\pi}{2}\sum_{n\geq 1}\frac{|v_n|^2}{n+1}(\leq 1)\]
\begin{proof}
By Lemma~\ref{lem_equicontinuity}, the sequence
\((v^{(k)})_{k\in\N}\) is locally equicontinuous. By the Arzelà--Ascoli theorem we may assume (up to the extraction of some sub-sequence) that it converges in \(\mathcal{C}^0_\loc(\B_1)\) to its weak-\(L^2(\B_1)\) limit \(v\).

Since the functions \(v^{(k)}\) are harmonic in \(\{v^{(k)}\neq 0\}\), this implies by elliptic regularity that the  convergence of \(v^{(k)}\) to \(v\) is strong in \(\mathcal{C}^2_\loc(\B_1\cap\{v\neq 0\})\).

As a consequence \(v\) satisfies
\[\Delta v=0,\ \Ho{v}=0\text{ in }\{v\neq 0\}.\]
If \(v\) is identically zero there is nothing to prove, so we assume from now on that  \(v\) is not identically zero.

Consider the function \(\log|v|\) with values in \([-\infty,+\infty)\). By continuity of \(v\), \(\log|v|\) is continuous. In the open set  \(\{v\neq 0\}\) we have
\begin{align*}
\Delta \log|v|&=\frac{v\cdot\Delta v}{|v|^2}+\frac{|v_x|^2+|v_y|^2}{|v|^2}-2\frac{(v\cdot v_x)^2+(v\cdot v_y)^2}{|v|^4}
\end{align*}
We replace \(\Delta v=0\), \(|v_y|^2=|v_x|^2\) and \((v\cdot v_x)^2+(v\cdot v_y)^2\leq |v|^2|v_x|^2\) (using the conformality of \(v\)), so \(\Delta \log|v|\geq 0\) in \(\{v\neq 0\}\). \(\log|v|\) is subharmonic in \(\B_1\setminus\{v=0\}\) and takes the value \(-\infty\) in \(\{v=0\}\), so it is subharmonic in \(\B_1\).

By~\cite[Th. 3.5.1]{Ransford95} the set \(\{v=0\}\) is polar. Using the result~\cite[Cor. 3.6.2]{Ransford95} on removable singularities, \(v\) is harmonic in \(\B_1\).

Since \(\Ho{v}=0\) in \(\{v\neq 0\}\) and \(\Ho{v}\) is continuous, we have  \(\Ho{v}=0\) everywhere: this concludes the proof.
\end{proof}
%Note: assuming we only know \(\Ho{u}\in L^q\) for some \(q>1\), I think this gives an alternative proof (with the stronger hypothesis \(\Delta(u^2)=2|\nabla u|^2\)) of Lipschitz continuity in the scalar (\(p=1\)) case. In this case the only harmonic conformal map \(\B_1\to \R\) vanishing at the origin is zero, so we get automatically that for any \(u\in\mathcal{S}\) with sufficiently large norm, we have
%\[\Norm{u_\frac{1}{2}}\leq \frac{1}{2}\Norm{u}.\]
%which implies that \(\Norm{u_r}\lesssim 1+\Norm{u}\) for all \(r\in (0,1)\).

\begin{definition}
For any \(u\in L^2(\B_1,\R^p)\setminus\{0\}\), we write
\[\delta(u)=\inf_{c\in \mathcal{C}}\frac{\Normr{ u-c}}{\Norm{u}}.\]
where \(\mathcal{C}\) is the set of linear, conformal maps \(\R^2\to\R^p\) i.e. functions of the form $z\mapsto \Re(c_1z)$ where  $c_1\in\C^p$ satisfies $c_1\cdot c_1=0$.
For any \(r\in (0,1)\), we write
\[u_r(z)=r^{-1}u(rz)\]
\end{definition}
We prove that any sufficiently large element of $\mathcal{S}$, which is almost linear in the annulus $\B_1\setminus\B_{\frac{1}{2}}$, is also almost linear in $\B_\frac{1}{2}$.
\begin{lemma}\label{lem_nonvanish}
For any \(\eta>0\), there exists \(M_0(\eta),\delta_0(\eta)>0\) such that the following holds. Let \(u\in\mathcal{S}\) be such that \(\Norm{u}\geq M_0(\eta)\) and \(\delta(u)\leq\delta_0(\eta)\). Let \(c\in\mathcal{C}\) such that \(\delta(u)\) is reached for the map \(c\). Then
\[\Vert u-c\Vert_{L^\infty(\B_\frac{1}{2})}\leq\eta\Norm{u}.\]
In particular, there exists a universal constant \(b>0\) such that \(u\) does not vanish in \(\B_\frac{1}{2}\setminus \B_{b\eta}\).
\end{lemma}

\begin{proof}
Assume this is not true, so there exists some sequence \((u^{(k)})_{k\in\N^*}\) in \(\mathcal{S}\) such that \(\Norm{u^{(k)}}\to +\infty\), \(\delta(u^{(k)})\to 0\) (we denote \(c^{(k)}\) the associated linear conformal map), and the conclusion does not hold. Let
\[v^{(k)}=\Norm{u^{(k)}}^{-1}u^{(k)},\ d^{(k)}=\Norm{u^{(k)}}^{-1}c^{(k)}.\]
By Lemma~\ref{lem_compactness}, we may assume (up to extracting some sub-sequence) that \(v^{(k)}\) converges locally uniformly to some harmonic, conformal limit \(v\) such that
\[\Norm{v}\leq 1,\ v(0)=0.\]
Since \(\Normr{v^{(k)}-d^{(k)}}\to 0\), then \(d^{(k)}\) converges to \(v\) in \(\B_1\setminus \B_{\frac{1}{2}}\). Thus \(v\) is \(1\)-homogeneous in \(\B_1\setminus\B_\frac{1}{2}\). By harmonicity of \(v\), \(v\) is \(1\)-homogeneous in \(\B_1\) and
\[\Vert v^{(k)}-d^{(k)}\Vert_{\mathcal{C}^0(\B_\frac{1}{2})}\underset{k\to +\infty}{\longrightarrow}0.\]
This is a contradiction with our original assumption, which concludes the proof of the first point. The second point is a consequence of the following observation: for any \(u\in\mathcal{S}\) satisfying \(\Norm{u}\geq M_0(\eta)\), \(\delta(u)\leq \delta_0(\eta)\), for any \(z\in\B_\frac{1}{2}\), we have
\begin{align*}
|u(z)|&\geq |c(z)|-\Vert u-c\Vert_{L^\infty(\B_\frac{1}{2})}\\
&=\sqrt{\frac{2}{\pi}} \Norm{c}|z|-\eta \Norm{u}
\end{align*}
Since \(\Norm{c}\geq  \Norm{u}-\Norm{u-c}\geq \frac{1}{2}\Norm{u}\) (under the assumption that \(\eta\) and \(\delta(u)\) are chosen sufficiently small since otherwise the second conclusion is vacuously true), this concludes the proof.
\end{proof}

Next we prove that any element of $\mathcal{S}$ that is sufficiently large must satisfy a mild decay estimate.

\begin{lemma}\label{lem_firstdecay}
For all \(\eps\in (0,1)\), there exists \(M_1(\eps)>0\) such that the following holds. Let \(u\in \mathcal{S}\) be such that \(\Norm{u}\geq M_1(\eps)\), then
\[\Normr{u_{\frac{1}{2}}}\leq \left(1+\eps\right)\Normr{u}.\]
\end{lemma}
\begin{proof}
Consider some sequence of functions \(u^{(k)}\in \mathcal{S}\) such that \(\Norm{u^{(k)}}\to +\infty\) and \(u^{(k)}\) does not satisfy the conclusion. By Lemma~\ref{lem_compactness}, the sequence  \(v^{(k)}=\Norm{u^{(k)}}^{-1}u^{(k)}\) converges (up to extracting some subsequence) weakly in \(L^2(\B_1)\), strongly in \(\mathcal{C}^0_\loc(\B_1)\), to some harmonic conformal limit \(v\) such that \(v(0)=0\) and \(\Norm{v}\leq 1\). We write
\[v(z)=\sum_{n\geq 1}\Re(v_n z^n)\]
for some sequence of coefficients \(v_n\in\C^p\). Then
\begin{align*}
\Normr{v}^2&=\frac{\pi}{2}\sum_{n\geq 1}\frac{(1-4^{-n-1})|v_n|^2}{n+1}\\
\Normr{v_\frac{1}{2}}^2&=\frac{\pi}{2}\sum_{n\geq 1}\frac{(1-4^{-n-1})|v_n|^2}{n+1}4^{1-n}
\end{align*}
Thus $\Normr{v_\frac{1}{2}}\leq \Normr{v}$. However,
\[\Normr{v^{(k)}_{\frac{1}{2}}}\geq (1+\eps)\Normr{v^{(k)}}\]
by assumption, which implies (by weak $L^2(\B_1)$ convergence and strong local convergence) $\Normr{v_\frac{1}{2}}\geq(1+\eps) \Normr{v}$. This implies $v=0$, which is in contradiction with our initial assumption.
\end{proof}

Here we give an improvement of the  previous lemma for sufficiently large elements of $\mathcal{S}$ that are bounded away from linear conformal maps.
\begin{lemma}\label{lem_seconddecay}
For all \(\delta\in \left(0,1\right)\), there exists \(M_2(\delta),\eps_2(\delta)>0\) such that the following holds. Let \(u\in \mathcal{S}\) be such that \(\Norm{u}\geq M_2(\delta)\), \(\delta(u)\geq\delta\), then
\[\Norm{u_{\frac{1}{2}}}\leq \left(1-\eps_2(\delta)\right)\Norm{u}.\]
\end{lemma}
\begin{proof}
Consider some sequence of functions \(u^{(k)}\in \mathcal{S}\) such that
\[\Norm{u^{(k)}}\to +\infty,\ \liminf_{k\to \infty}\frac{\Norm{u_\frac{1}{2}^{(k)}}}{\Norm{u^{(k)}}}\geq  1,\ \delta(u^{(k)})\geq \delta.\]
By Lemma~\ref{lem_compactness}, the sequence  \(v^{(k)}=\Norm{u^{(k)}}^{-1}u^{(k)}\) converges (up to extracting some subsequence) weakly in \(L^2(\B_1)\), strongly in \(\mathcal{C}^0_\loc(\B_1)\), to some harmonic conformal limit \(v\) such that \(v(0)=0\) and \(\Norm{v}\leq 1\). We write
\[v(z)=\sum_{n\geq 1}\Re(v_n z^n)\]
for some sequence of coefficients \(v_n\in\C^p\), then
\begin{align*}
\Norm{v}^2&=\frac{\pi}{2}\sum_{n\geq 1}\frac{|v_n|^2}{n+1}\\
\Norm{v_\frac{1}{2}}^2&=\frac{\pi}{2}\sum_{n\geq 1}\frac{|v_n|^2}{n+1}4^{1-n}
\end{align*}
We deduce
\[1\geq \Norm{v}\geq\Norm{v_\frac{1}{2}}=\lim_{k\to +\infty}\Norm{v^{(k)}_\frac{1}{2}}\geq 1.\]
By the convergence of the $L^2(\B_1)$ norm, the convergence of \(v^{(k)}\) to \(v\) is strong in \(L^2(\B_1)\). This implies \(\delta(v)=\lim_{k\to +\infty}\delta(v^{(k)})\geq \delta\), at the same time the chain of inequalities above implies that \(v\) is linear, so \(\delta(v)=0\). This is a contradiction, which proves the result.
\end{proof}

We now list several elementary computations on harmonic functions, based on their Laurent decomposition.
\begin{lemma}\label{lem_decompisition}
Let \(u\in H^1(\B_{\rho_2}\setminus\B_{\rho_1},\R^p)\) for some \(\rho_1<\rho_2\) such that \(\Delta u=0\) in \(\B_{\rho_2}\setminus\B_{\rho_1}\). Then \(u\) may be decomposed as an infinite sum
\[u(re^{i\theta})=\sum_{n\in\Z}u_n(r)e^{i n \theta}\]
for \(r\in (\rho_1,\rho_2)\), where \(u_n(r):=\fint_{0}^{2\pi}u(re^{i\theta})e^{-in\theta}d\theta\) is of the form
\[u_n(r)=a_n r^{|n|}+b_n r^{-|n|}\text{ for }n\neq 0,\ u_0(r)=a_0+b_0\log(r)\]
for some sequence  of coefficients \((a_n)_{n\in\Z}\), \((b_n)_{n\in\Z}\) with values in \(\C^p\). Moreover, for any \(r>0\) such that \(\left[\frac{r}{2},r\right]\subset[\rho_1,\rho_2]\), we have
\[\Normr{u_r}^2=Q_0(a_0+\log(r)b_0,b_0)r^{-2}+\sum_{n\neq 0}Q_n(a_n r^{|n|-1},b_n r^{-|n|-1})\]
where \((Q_n)\) is a sequence of quadratic forms such that
\[\frac{|x|^2+4^{|n|} |y|^2}{|n|+1}\lesssim Q_n(x,y)\lesssim \frac{|x|^2+4^{|n|} |y|^2}{|n|+1}.\]
\end{lemma}
\begin{proof}
The form of \(u_n(r)\) is a direct consequence of the associated differential equation \(u_n''(r)+r^{-1}u_n'(r)-n^2r^{-2} u_n(r)=0\). Then an explicit computation gives, for any \(n\in\Z\setminus \{-1,0,1\}\),
\[Q_n(x,y)=\pi\left(\frac{1-4^{-|n|-1}}{|n|+1}|x|^2+\frac{4^{|n|-1}-1}{|n|-1}|y|^2+\frac{3}{2}\Re(x\ov{y})\right)\]
which concludes the proof, since the third term is negligible compared to the sum of the first two terms as \(n\to \pm\infty\).
\end{proof}

The next lemma is a purely technical point that is used in the proof of Lemma~\ref{lem_analysisunstablemodes}.
\begin{lemma}\label{lem_control_problematic_mode}
There exists \(\delta_3\in (0,1)\) such that for any \(\eps\in (0,1)\), there exists \(M_3(\eps)>0\) such that the following holds.  Let \(u\in\mathcal{S}\) be such that  \(\Norm{u}\geq M_3(\eps)\) and \(\delta(u)\leq \delta_3\). Then \(u\) does not vanish on \(\B_\frac{1}{2}\setminus\B_{\frac{1}{4}}\), and writing \(u_{\pm 1} (r)=a_{\pm 1} r+b_{\pm 1} r^{-1}\) the first Fourier coefficients of \(\theta\mapsto u(re^{i\theta})\) for \(r\in \left(\frac{1}{4},\frac{1}{2}\right)\), we have
\[|b_1|+|b_{-1}|\leq \eps \Norm{u}\]
\end{lemma}
\begin{proof}
We let  \(\delta_3,M>0\) be such that if \(\delta(u)\leq \delta_3\) and \(\Norm{u}\geq M\), then by Lemma~\ref{lem_nonvanish} we have that \(u\) does not vanish in \(\B_{\frac{1}{2}}\setminus\B_{\frac{1}{4}}\).

Let now \(u^{(k)}\) be a sequence of functions such that  \(\delta(u^{(k)})\leq\delta_3\), \(\Norm{u^{(k)}}\to +\infty\), and the conclusion is not satisfied. We let
\[v^{(k)}=\Norm{u^{(k)}}^{-1}u^{(k)}\]
Write \(v_n^{(k)}(r)\) the \(n\)-th Fourier coefficient of \(\theta\mapsto v^{(k)}(r e^{i\theta})\): there exists \((a_{\pm 1}^{(k)},b_{\pm 1}^{(k)})\) such that
\[v_{\pm 1}^{(k)}(r)=a_{\pm 1}^{(k)}r+b_{\pm 1}^{(k)}r^{-1}\]
for all \(r\in \left(\frac{1}{4},\frac{1}{2}\right)\), and we assume that \(|b_{ 1}^{(k)}|+|b_{-1}^{(k)}|\geq\eps\).

By Lemma~\ref{lem_compactness}, there exists some extraction of \((v^{(k)})_{k\in\N}\) that converges in \(\mathcal{C}^0_\loc(\B_1)\) to some harmonic, conformal limit \(v\) such that \(v(0)=0\) and \(\Norm{v}\leq 1\). In particular, \(v_{\pm 1}(r)=a_{\pm 1}r\) for some \(a_{\pm 1}\in\C^p\), so
\[b_{\pm 1}^{(k)}\underset{k\to+\infty}{\longrightarrow} 0\]
Thus for any large enough \(k\) we obtain a contradiction.
\end{proof}
The next lemma is the central point of the  proof, where we deal with almost linear elements of $\mathcal{S}$.
\begin{lemma}\label{lem_analysisunstablemodes}
There exist \(M_4,C_4>0\), \(\delta_4\in (0,1)\) such that the following holds. Let  \(u\in\mathcal{S}\) be such that \(\Norm{u}\geq M_4\) and \(\delta(u)\leq \delta_4\). Then \(u\) does not vanish in \(\B_{\frac{1}{2}}\setminus \{0\}\) and
\[\Norm{u_r}\leq  C_4\Norm{u}\]
for all \(r\in (0,1)\).
\end{lemma}
\begin{proof}
Let \(k\in\N^*\) be some integer that will be fixed (independently of other constants) at the very end of the proof.
Let
\begin{align*}
\delta&=\min\left(\frac{1}{5},\delta_3,\delta_0\left(b^{-1}2^{-k-1}\right)\right)\\
\delta_4&=\min\left(\frac{\delta}{k},\delta_0\left(\frac{\delta}{k}\right)\right)\\
M_4&=2\max\left(M_0\left(b^{-1}2^{-k-1}\right),M_0\left(\frac{\delta}{k}\right),M_1\left(\frac{\delta^2}{k}\right),M_3\left(2^{-2k}\delta^2\right)\right)
\end{align*}
Here \(b\) is the constant from Lemma~\ref{lem_nonvanish}. In other words, for any \(u\in\mathcal{S}\) such that \(\Norm{u}\geq \frac{1}{2}M_4\), we have the following properties:
\begin{align}
&\Normr{u_{\frac{1}{2}}}\leq \left(1+\frac{\delta^2}{k}\right)\Normr{u}\label{eq_aux_decay}\\
&\delta(u)=:\frac{\Normr{u-c}}{\Norm{u}}\leq \delta_4 \Rightarrow \frac{\Vert u-c\Vert_{L^\infty(\B_\frac{1}{2})}}{\Norm{u}}\leq \frac{\delta}{k}\label{eq_aux_controldelta}\\
&\delta(u)\leq \delta \Rightarrow \Delta u|_{\B_\frac{1}{2}\setminus\B_{\frac{1}{2^{k+1}}}}=0\text{ and } |b_1|+|b_{-1}|\leq 2^{-2k}\delta^2\Norm{u}\label{eq_aux_controlfirstmode}
\end{align}
where \(b_{\pm 1}\) are the coefficients in the Fourier decomposition \(u_{\pm 1}(r)=a_{\pm 1}r+b_{\pm 1}r^{-1}\) for \(r\in\left(\frac{1}{2^{k+1}},\frac{1}{2}\right)\).\bigbreak

Consider from now on \(u\in\mathcal{S}\) such that \(\Norm{u}\geq M_4\) and \(\delta(u)\leq\delta_4\). Fix \(c\in\mathcal{C}\) such that \(\Normr{u-c}=\delta(u)\Norm{u}\).
Since \(\delta_4\leq\delta\), by Lemma~\ref{lem_nonvanish} \(u\) does not vanish on \(\B_{\frac{1}{2}}\setminus\B_{\frac{1}{2^{k+1}}}\). We let \((a_n,b_n)\) be the coefficients associated to the Fourier decomposition of \(u\) in this set, meaning that
\[u(r e^{i\theta})=a_0+b_0\log(r)+\sum_{n\neq 0}\left(a_n r^{|n|}+b_n r^{-|n|}\right)e^{in\theta}.\]
We have by~\eqref{eq_aux_controldelta} that
\[\Normr{u_{\frac{1}{2}}-c}\lesssim \frac{\delta}{k}\Norm{u}.\]
By \(L^2(\B_1\setminus\B_\frac{1}{2})\)-orthogonal projection of \(u_{\frac{1}{2}}\) on \(\mathrm{Span}_{\C^p}\left(z,\ov{z},z^{-1},\ov{z}^{-1}\right)\) and since \(c\in\mathrm{Span}_{\C^p}(z,\ov{z})\), we have
\[\Normr{a_1 z+a_{-1}\ov{z}-c}\lesssim \frac{\delta}{k}\Norm{u}.\]
Assuming that \(k\) is sufficiently large, we have
\[\Normr{a_1 z+a_{-1}\ov{z}-c}\leq \frac{\delta}{4}\Norm{u}\text{ and }\Normr{u_{\frac{1}{2}}-a_1 z-a_{-1}\ov{z}}\leq \frac{\delta}{4}\Norm{u}.\]
We let \(\rho\in\left[0,\frac{1}{2}\right)\) be the smallest value such that
\[\Normr{u_{s}-a_1 z-a_{-1}\ov{z}}\leq \frac{\delta}{4}\Norm{u},\ \forall s\in\left(\rho,\frac{1}{2}\right].\]
For every \(s\in\left(\rho,\frac{1}{2}\right]\) we have
\begin{align*}
\Norm{u_s}&\geq \Normr{u_s}\geq \Normr{c}-\delta\Norm{u}\\
&=\sqrt{\frac{15}{16}}\Norm{c}-\delta\Norm{u}\geq \sqrt{\frac{15}{16}}\Norm{u}-2\delta\Norm{u}\\
&\geq \left(\sqrt{\frac{15}{16}}-2\delta\right)\Norm{u}\geq \frac{1}{2}\Norm{u}\text{ since }\delta\leq\frac{1}{5}
\end{align*}
so $\Norm{u_s}\geq \frac{M_4}{2}$, and
\begin{align*}
\Normr{u_s-c}&\leq \Normr{u_s-a_1 z-a_{-1}\ov{z}}+\Normr{a_1 z+a_{-1}\ov{z}-c}\\
&\leq\frac{\delta}{2}\Norm{u}\leq \delta \Norm{u_s}
\end{align*}
so $\delta(u_s)\leq\delta$. By application of ~\eqref{eq_aux_controlfirstmode} to $u_s$, $u$ is harmonic in $\B_\frac{1}{2}\setminus\B_{2^{-k-1}\rho}$ (or $\B_{\frac{1}{2}}\setminus\{0\}$ when $\rho=0$).

Our goal is to prove that \(\rho=0\). Indeed in this case \(u\) is harmonic in \(\B_{\frac{1}{2}}\setminus\{0\}\), which implies (by removal of isolated singularities) the harmonicity of \(u\) in \(\B_\frac{1}{2}\) as well as the estimate on  \(\Norm{u_r}\) by classical elliptic regularity.

We proceed by contradiction. Assume that \(\rho\) is positive, so that
\[\Normr{u_{\rho}-a_1 z-a_{-1}\ov{z}}= \frac{\delta}{4}\Norm{u}.\]
%We have \(\Norm{u_\rho}\geq \frac{M_4}{2}\) and \(\delta(u_\rho)\leq \delta\) from the previous computation, so~\eqref{eq_aux_controlfirstmode} applies.
We estimate the non-singular harmonic modes at the radius \(\frac{1}{2}\):
\begin{align*}
\left(\frac{\delta}{k}\Norm{u}\right)^2&\gtrsim \Normr{u_{\frac{1}{2}}-a_1 z-a_{-1}\ov{z}}^2\\
&=Q_0(a_0+\log(\frac{1}{2})b_0,b_0)2^{2}+Q_1(0,b_1)2^{4}+Q_{-1}(0,b_{-1})2^4\\
&+\sum_{n\neq 0,\pm 1}Q_n(a_n 2^{1-|n|},b_n 2^{1+|n|})\\
&\gtrsim \sum_{n\neq 0,\pm 1}\frac{4^{-(|n|-1)}}{|n|}|a_n|^2.
\end{align*}
Then, at the radius \(\rho\) we can estimate the singular harmonic modes as follows.
\begin{align*}
&\left(\frac{\delta}{4}\Norm{u}\right)^2= \Normr{u_{\rho}-a_1 z-a_{-1}\ov{z}}^2\\
&=Q_0(a_0+\log(\rho)b_0,b_0)\rho^{-2}+Q_1(0,b_1)\rho^{-4}+Q_{-1}(0,b_{-1})\rho^{-4}+\sum_{n\neq 0,\pm 1}Q_n(a_n \rho^{|n|-1},b_n \rho^{-|n|-1})\\
&\lesssim   \sum_{n\neq 0,\pm 1}\frac{\rho^{2(|n|-1)}}{|n|}|a_n|^2+\left(\rho^{-2}\left(|a_0+\log(\rho)b_0|^2+|b_0|^2\right)+\sum_{n\neq 0}\frac{4^{|n|}}{|n|}\rho^{-2(|n|+1)}|b_n|^2\right)\\
&\lesssim \left(\frac{\delta}{k}\Norm{u}\right)^2+\left(\rho^{-2}\left(|a_0+\log(\rho)b_0|^2+|b_0|^2\right)+\sum_{n\neq 0}\frac{4^{|n|}}{|n|}\rho^{-2(|n|+1)}|b_n|^2\right)
\end{align*}
Thus for a sufficiently large \(k\) (which we assume), we have
\[\rho^{-2}\left(|a_0+\log(\rho)b_0|^2+|b_0|^2\right)+\sum_{n\neq 0}\frac{4^{|n|}}{|n|}\rho^{-2(|n|+1)}|b_n|^2\gtrsim\left(\delta\Norm{u}\right)^2.\]
We recall that $u$ is harmonic in $\B_{\frac{1}{2}}\setminus\B_{2^{-k-1}\rho}$, so for any $\ell\in\{0,1,\hdots,k\}$ we have
\begin{align*}
\Normr{u_{2^{-\ell}\rho}-a_1 z-a_{-1}\ov{z}}^2
&=Q_0(a_0+\log\left(2^{-\ell}\rho\right)b_0,b_0)\left(2^{-\ell}\rho\right)^{-2}+Q_1(0,b_1)\left(2^{-\ell}\rho\right)^{-4}\\
&+Q_{-1}(0,b_{-1})\left(2^{-\ell}\rho\right)^{-4}
+\sum_{n\neq 0,\pm 1}Q_n\left(a_n (2^{-\ell}\rho)^{|n|-1},b_n (2^{-\ell}\rho)^{-|n|-1}\right)\\
&\gtrsim (1+\ell)^{-2}2^{2\ell}\rho^{-2}\left(|a_0+\log(\rho)b_0|^2+|b_0|^2\right)+\sum_{n\neq 0}\frac{4^{|n|}}{|n|}\left(2^{-\ell}\rho\right)^{-2(|n|+1)}|b_n|^2\\
&\gtrsim (1+\ell)^{-2}2^{2\ell}(\delta\Norm{u})^2
\end{align*}
By~\eqref{eq_aux_controlfirstmode} we have \(|b_1\rho^{-2}|\leq 2^{-2k}\delta^2\Norm{u_\rho}\lesssim  2^{-2k}\delta^2\Norm{u}\), so for a sufficiently large \(k\) we have
\begin{align*}
&\Normr{u_\rho}^2-\Normr{u_{\rho}-a_1 z-a_{-1}\ov{z}}^2-\Normr{a_1 z+a_{-1}\ov{z}}^2\\
&\lesssim |a_1|\cdot |b_1\rho^{-2}|+|a_{-1}|\cdot |b_{-1}\rho^{-2}|\lesssim 2^{-2k}\delta^2\Norm{u}^2
\end{align*}
and for all $\ell\in\{0,1,2,\hdots,k\}$,
\begin{align*}
&\Normr{u_{2^{-\ell}\rho}}^2-\Normr{u_{2^{-\ell}\rho}-a_1 z-a_{-1}\ov{z}}^2-\Normr{a_1 z+a_{-1}\ov{z}}^2\\
&\gtrsim-|a_1|\cdot \left|b_1(2^{-\ell}\rho)^{-2}\right|-|a_{-1}|\cdot \left|b_{-1}(2^{-\ell}\rho)^{-2}\right|\gtrsim -\delta^2\Norm{u}^2.
\end{align*}
Assuming \(k\) is sufficiently large, and reminding that \(\Norm{u}\lesssim \Normr{a_1 z+a_{-1}\ov{z}}\lesssim \Norm{u}\), we obtain some constants \(a\lesssim 1\), \(a'\gtrsim 1\) such that
\[\Normr{u_\rho}^2\leq  \Normr{a_1 z+a_{-1}\ov{z}}^2\left(1+a\delta^2\right)\]
and for all $\ell\in\{0,1,\hdots,k\}$:
\[\Normr{u_{2^{-\ell}\rho}}^2\geq  \Normr{a_1 z+a_{-1}\ov{z}}^2\left(1+a'(1+\ell)^{-2}2^{2\ell}\delta^2\right)\]
Since $\Normr{a_1z+a_{-1}\ov{z}}\geq \frac{M_4}{2}$ we may apply the growth estimate~\eqref{eq_aux_decay} to \(u_{2^{-\ell}\rho}\) for \(\ell=0,1,\hdots,k-1\), which gives:
\[\Normr{u_{2^{-k}\rho}}\leq \left(1+\frac{\delta^2}{k}\right)^k\Normr{u_\rho}.\]
Using the elementary estimate \(\left(1+\frac{\delta^2}{k}\right)^k\leq e^{\delta^2}\), we obtain
\[1+a'(1+k)^{-2}2^{2k}\delta^2\leq e^{2\delta^2}(1+a\delta^2).\]
For any sufficiently large $k\in\N^*$, this does not hold for any \(\delta\in (0,1)\). Fixing such large enough value for \(k\), we arrive at a contradiction from the fact that \(\rho\) is positive. This concludes the proof.
\end{proof}

\begin{proof}[Proof of Theorem~\ref{th_crit_vectorial}]
Since \(u\) is harmonic in its support, using classical elliptic regularity in the interior of the support it is sufficient to prove the following: assume \(u\in\mathcal{S}\), then
\[\Norm{u_r}\lesssim 1+ \Norm{u}\]
for all \(r\in (0,1)\). Let
\[M=\max(M_4,M_2(\delta_4)),\ \eps=\eps_2(\delta_4)\]
Then for every \(r\in (0,1)\), by Lemmas~\ref{lem_seconddecay} and~\ref{lem_analysisunstablemodes}, one of the three following statements holds:
\begin{align*}
\text{either }& \Norm{u_r}\leq M\text{ or }\Norm{u_{\frac{r}{2}}}\leq (1-\eps)\Norm{u_r}\text{ or }\sup_{0<s<r}\Norm{u_s}\leq C_4\Norm{u_r}.
\end{align*}
Let \(\rho\in [0,1]\) be the smallest radius such that either of the first two conditions holds for all \(r\in (\rho,1)\). When \(\rho=1\), we obtain the result using the third condition. Assume now that \(\rho<1\).

For all \(r\in (\rho,1]\) we have
\[\Norm{u_{\frac{r}{2}}}\leq 4M+(1-\eps)\Norm{u_r}.\]
This implies
\[\Norm{u_r}\lesssim 1+\Norm{u},\forall r\in (\rho,1].\]
When \(\rho=0\) this concludes the proof. When \(\rho>0\) we then have by the third condition, for all \(r\leq \rho\),
\[\Norm{u_r}\leq  C_4\Norm{u_{\rho}}\lesssim 1+\Norm{u}\]
which conclude the proof.
\end{proof}
\subsection{Non-degeneracy for stationary solutions}

It is not possible to prove a non-degeneracy result as strong as in the minimizing case, as is demonstrated by the following example: Let \(u(z)=|\mathrm{Im}(z^{3/2})|\), then \(u\) is a stationary solution of the Bernoulli problem, \(r^{-1}u(r\cdot)\) converges to \(0\) as \(r\to 0\), but \(u\) does not vanish in any neighborhood of the origin.

Moreover, similarly to the Lipschitz regularity for the vector-valued stationary solutions, one may also construct counterexamples to non-degeneracy relying on the fact that $u$ does  not satisfy an equation on its support; for any \(\eps>0\), let \(u(x,y)=\left(\eps -|x|\right)_+\), then \(u\) is a stationary solution of the Bernoulli problem, \(\Vert u\Vert_{H^1(\B_1)}+|\{u\neq 0\}|\) is arbitrarily small for \(\eps\to 0\), and \(u\) does not vanish near the origin.

Thus a non-degeneracy result must involve both the support being small and \(u\) satisfying an equation on its support in the sense of \eqref{eq_structure_weak}, even in the scalar-valued case. In~\cite[Lem. 9.2]{KW25}, such a result is established for their notion of variational solutions, in any dimension. The proof is based on a blow-up argument, deducing information from the blow-up through a monotonicity formula. We state the same result in our two-dimensional case, and give an alternative proof which extends to the vector-valued case and may have independent interest.
\begin{proposition}\label{prop_crit_nondeg}
Let $d=2$, $p\in\N^*$, there exists \(c_p>0\) such that the following holds. Let \(u\in H^1(\B_1,\R^p)\) be a stationary solution of the Bernoulli problem such that $u_k\Delta u_l=0$ for all $k,l\in\{1,\hdots,p\}$, and such that \(u(0)\neq 0\). Then
\[\int_{\B_1}\left(|\nabla u|^2+\chi_{u\neq 0}\right)\geq c_p.\]
\end{proposition}

We obtain Proposition~\ref{prop_crit_nondeg} as a consequence of the following lemmas.

\begin{lemma}\label{lem_crit_nondeg_1}
There exist \(C_4,\eps_0>0\) such that the following holds. Let \(u\in H^1(\B_1,\R^p)\) be a stationary solution of the Bernoulli problem that satisfies
\[\int_{\B_1}\left(|\Ho{u}|+\chi_{u\neq 0}\right)<\eps_0.\]
Then
\[|\{u\neq 0\}\cap\B_{\frac{1}{2}}|\leq C_4\Vert\Ho{u}\Vert_{L^\infty(\B_1)} \int_{\B_1}|\Ho{u}|.\]
\end{lemma}

\begin{proof}
We write \(\ov{\mathcal{B}}:L^2(\R^2)\to L^2(\R^2)\) the operator defined by
\[\widehat{\ov{\mathcal{B}}f}(\xi,\zeta)=\frac{\xi+\iu \zeta}{\xi-\iu \zeta}\widehat{f}(\xi,\zeta)\]
i.e. \(\ov{\mathcal{B}}\) is the inverse of the operator \(\mathcal{B}\) from Definition~\ref{def_B}, which shares the same property as \(\mathcal{B}\) stated in Lemma~\ref{lem_B} with \(\partial_z\ov{\mathcal{B}}[f]=\partial_{\ov{z}}f\) instead. Since \(u\) is a stationary solution of the Bernoulli problem, we have
\[4\partial_{\ov{z}}\Ho{u}-\partial_{z}\chi_{u\neq 0}=0.\]
We let
\[\Psi=\ov{\mathcal{B}}\left[4\Ho{u}\chi_{\B_1}\right]-\chi_{u\neq 0},\]
then \(\partial_{z}\Psi=0\) in \(\B_1\), so \(\Psi\) is antiholomorphic in \(\B_1\). We have
\begin{align*}
\Vert \Psi\Vert_{L^{1,\infty}(\B_{1})}&\lesssim \left\Vert \ov{\mathcal{B}}\left[4\Ho{u}\chi_{\B_1}\right]\right\Vert_{L^{1,\infty}(\B_{1})}+\Vert \chi_{u\neq 0}\Vert_{L^{1,\infty}(\B_{1})}\\
&\lesssim \Vert \Ho{u}\Vert_{L^{1}(\B_{1})}+\Vert \chi_{u\neq 0}\Vert_{L^{1}(\B_{1})}\text{ by Lem. }\ref{lem_B}\\
&\lesssim\int_{\B_1}\left(|\Ho{u}|+\chi_{u\neq 0}\right),
\end{align*}
so by Lemma~\ref{lem_L1infty}, we have
\[
\Vert \Psi\Vert_{L^{\infty}(\B_{\frac{1}{2}})}\lesssim\int_{\B_1}\left(|\Ho{u}|+\chi_{u\neq 0}\right).\]
If the right-hand side is sufficiently small, which we assume, then we have \(\Vert\Psi\Vert_{L^\infty(\B_{\frac{1}{2}})}\leq \frac{1}{2}\), so \(\chi_{\B_{\frac{1}{2}}\cap\{u\neq 0\}}\leq 2|\Psi+\chi_{\{u\neq 0\}}|\). As a consequence,
\begin{align*}
|\{u\neq 0\}\cap\B_{\frac{1}{2}}|^\frac{1}{2}&\leq 2 \Vert \Psi+\chi_{u\neq 0}\Vert_{L^2(\B_\frac{1}{2})}=8\left\Vert \ov{\mathcal{B}}\left[\Ho{u}\chi_{\B_1}\right]\right\Vert_{L^2(\B_\frac{1}{2})}\\
&\lesssim \Vert \Ho{u}\Vert_{L^2(\B_{1})}&\text{ by Lemma }\ref{lem_B}\\
&\leq \Vert \Ho{u}\Vert_{L^\infty(\B_1)}^\frac{1}{2}\Vert \Ho{u}\Vert_{L^1(\B_1)}^\frac{1}{2}.
\end{align*}
\end{proof}

The next lemma is a more classical consequence of the Caccioppoli inequality.
\begin{lemma}\label{lem_crit_nondeg_2}
There exists \(C_5>0\) such that the following holds. Let \(u\in H^1(\B_1,\R^p)\) such that  $u_k\Delta u_l=0$ for all $k,l$. Then
\[\int_{\B_\frac{1}{2}}|\nabla u|^2\leq C_5|\{u\neq 0\}\cap\B_1|^\frac{1}{2}\int_{\B_1}|u|^2.\]
\end{lemma}

\begin{proof}
Our assumption implies that for all $\varphi\in\mathcal{C}^\infty_c(\B_1,\R)$, we have
\[\int_{\B_1}\varphi^2|\nabla u|^2\leq 4\int_{\B_1}|\nabla\varphi|^2|u|^2\]
which implies the  following Caccioppoli inequality.
\[\int_{\B_{s}}|\nabla u|^2\lesssim \frac{1}{(r-s)^2} \int_{\B_{r}}|u|^2\]
for any \(0<s<r<1\). Using the embedding \(H^1(\B_{\frac{3}{4}})\hookrightarrow L^4(\B_{\frac{3}{4}})\), we have
\begin{align*}
\int_{\B_{\frac{1}{2}}}|\nabla u|^2&\lesssim \int_{\B_{\frac{3}{4}}}|u|^2\lesssim |\{u\neq 0\}\cap\B_1|^\frac{1}{2}\left(\int_{\B_{\frac{3}{4}}}|u|^4\right)^\frac{1}{2}\\
&\lesssim |\{u\neq 0\}\cap\B_1|^\frac{1}{2}\int_{\B_{\frac{3}{4}}}\left(|\nabla u|^2+|u|^2\right)\\
&\lesssim |\{u\neq 0\}\cap\B_1|^\frac{1}{2}\int_{\B_{1}}|u|^2.
\end{align*}
\end{proof}

\begin{proof}[Proof of Proposition~\ref{prop_crit_nondeg}]
Let \(u\) be a function that satisfies the hypothesis of Proposition~\ref{prop_crit_nondeg}, such that $\int_{\B_1}\left(|\nabla u|^2+\chi_{u\neq 0}\right)\leq 1$. By Lemma~\ref{lem_hopf_bounded}, we have
\[C_4\Vert\Ho{u}\Vert_{L^\infty(\B_\frac{1}{2})}\leq C_4'\]
for some universal $C_4'>0$. We let \(r_k=\frac{1}{2}4^{-k}\), assume that for some \(k\in\N\), we have
\[e_k:=\frac{1}{r_k^2}\int_{\B_{r_k}}\left(|\nabla u|^2+\chi_{u\neq 0}\right)\leq \min\left(\frac{\pi}{2},\eps_0\right).\]
Since \(|\{u\neq 0\}\cap\B_{r_k}|\leq \frac{1}{2}|\B_{r_k}|\), by the Poincaré inequality we have \(\fint_{\B_{r_k}}|u/r_k|^2\leq C_6\fint_{\B_{r_k}}|\nabla u|^2\) for some constant \(C_6>0\). Then combining Lemmas~\ref{lem_crit_nondeg_1} and~\ref{lem_crit_nondeg_2}, we get
\begin{align*}
\frac{1}{r_k^2}\int_{\B_{\frac{1}{2}r_k}}|\nabla u|^2&\leq C_5C_6 \left(\frac{1}{r_k^2}\int_{\B_{r_k}}\chi_{u\neq 0}\right)^\frac{1}{2}
\frac{1}{r_k^2}\int_{\B_{r_k}}|\nabla u|^2\\
\int_{\B_{r_{k+1}}}\chi_{u\neq 0}&\leq C_4'\int_{\B_{\frac{1}{2}r_k}}|\nabla u|^2
\end{align*}
so we obtain \(e_{k+1}\leq 16(4+C_4')C_5C_6e_k^{3/2}\). We fix
\[\eps:=\min\left(\frac{\pi}{8},\eps_0,\left(\frac{1}{32(4+C_4')C_5C_6}\right)^2\right)\]
Then, if $\int_{\B_1}\left(|\nabla u|^2+\chi_{u\neq 0}\right)\leq \frac{\eps}{4}$, then $e_0\leq\eps$ and for any \(k\in\N\), we have
\[\left(e_k\leq \eps\right)\Rightarrow \left(e_{k+1}\leq \frac{1}{2}e_k\right).\]
In particular, by continuity of \(u\) at the origin this implies \(u(0)=0\). This implies the result with $c_p=\frac{\eps}{4}$.
\end{proof}

\section{Applications to shape optimization}\label{sec_spectral}

For the sake of simplicity, we presented our result for energies with constant coefficients and no lower order perturbation. In this section, we extend Corollary~\ref{cor_Q} to several shape optimization problems coming from spectral geometry.

Let \(d\in\N^*\), \(n,m\in\N\) such that \(n>m\). Let \(Q_n\) (resp.\ \(Q_m\)) be a quadratic form on \((\R^d)^{\otimes n}\otimes\R^p\) (resp.\ \((\R^d)^{\otimes m}\otimes\R^p\)). Assume \(Q_n,Q_m\) both satisfy the Legendre--Hadamard condition~\eqref{eq_legendrehadamard}, and consider for any open set with finite measure \(\Om\subset\R^d\) the eigenvalue
\[\Lambda_1[\Om;Q_n/Q_m]:=\inf_{u\in H^n_0(\Om,\R^p)\setminus\{0\}}\frac{\int_{\Om}Q_n(\nabla^n u)}{\int_{\Om}Q_m(\nabla^m u)}.\]
Write \(\mathcal{A}\) the set of all open subsets of \(\R^d\) with measure \(1\).
\begin{proposition}\label{prop_eigenvalue}
Let \(Q_n\), \(Q_m\) be defined as above. Then the minimization problem
\[\inf_{\Om\in\mathcal{A}}\Lambda_1[\Om;Q_n/Q_m],\]
reaches its minimum, and any first eigenfunction \(u\) associated to some minimizer belongs to \(W^{n,\infty}(\R^d,\R^p)\). Moreover  any minimizer is bounded.
\end{proposition}

In general the first eigenvalue might not be simple: in this case the proof implies that the whole eigenspace belongs to \(W^{n,\infty}(\R^d,\R^p)\).
% This proposition includes the following eigenvalue problems.
% \begin{itemize}
% \item[i)] Lamé eigenvalue, corresponding to \(p=d\), \(n=1\), \(m=0\), \(Q_1(\nabla u)=2\mu |e(u)|^2+\lambda \mathrm{div}(u)^2\) and \(Q_0(u)=|u|^2\), for \(\lambda+2\mu>0\). This was studied in \cite{HLP24} under the stronger hypothesis \(\lambda+\mu>0\), in particular in \cite[Th. 1.2.]{HLP24} it is proved that in dimension \(d=2\), for \((\lambda,\mu)\) satisfying
% \[\mu>0,\ -\mu< \lambda\leq 4\mu,\]
% then the disk is not a minimizer.
% \item[ii) ]Clamped plate eigenvalue, corresponding to \(p=1\), \(n=2\), \(m=0\), \(Q_2(\nabla^2 u)=|\Delta u|^2\) and \(Q_0(u)=u^2\). For \(d\in\{2,3\}\), it is known from \cite{N95,AB95} that the optimal set is a ball and this is still an open problem in dimension \(d\geq 4\).
% \item[iii) ]Buckling eigenvalue, corresponding to \(p=1\), \(n=2\), \(m=1\), \(Q_2(\nabla^2 u)=|\Delta u|^2\) and \(Q_1(u)=|\nabla u|^2\). The minimality of the disk in dimension \(d=2\) is an open problem attributed to Pólya and Szegö, see \cite[Sec. 6.2]{LN25} for a detailed discussion on the state of this conjecture, which is slightly improved by proposition \ref{prop_eigenvalue}.
% \end{itemize}
\begin{proof}[Proof of Proposition~\ref{prop_eigenvalue}]
We have the following scaling property, for any \(\Om\) with finite area and any \(t>0\):
\[\Lambda_1[t\Om;Q_n/Q_m]=t^{2(m-n)}\Lambda_1[\Om;Q_n/Q_m].\]
We may relax this minimization problem as
\begin{align*}
\inf_{\Om\in\mathcal{A}}\Lambda_1[\Om;Q_n/Q_m]&=\inf_{\Om\subset\R^d:|\Om|\leq 1}\Lambda_1[\Om;Q_n/Q_m]\\
&=\inf_{\Om\subset\R^d:|\Om|\leq 1}\inf_{u\in H^n_0(\Om,\R^p)\setminus\{0\}}\frac{\int_{\Om}Q_n(\nabla^nu)}{\int_{\Om}Q_m(\nabla^m u)}\\
&\geq\inf\left\{\frac{\int_{\Om}Q_n(\nabla^nu)}{\int_{\Om}Q_m(\nabla^m u)},\ u\in H^n(\R^d,\R^p)\setminus\{0\}: |\{u\neq 0\}|\leq 1\right\}
\end{align*}
The existence of a minimizer for this is classical by the concentration-compactness principle in \(\R^d\) developed in~\cite{Lions85} (see~\cite[Th. 3.1]{BF17} for an adaptation to the first Laplace eigenvalue,~\cite[Th. 1]{HMP24} for the first Stokes eigenvalue, and~\cite[Th. 3]{L25} for the polyharmonic operator). The vanishing of the minimizing sequence does not occur, based on the following observation: for any such function \(u\in H^n(\R^d,\R^p)\), we may find some square \(K=z(u)+[0,1]^d\) (for \(z(u)\in\Z^d\)) such that
\[|K\cap \{u\neq 0\}|\geq C_{Q_n,Q_m}\left(\frac{\int_{\R^d}Q_n(\nabla^n u)}{\int_{\R^d}Q_m(\nabla^m u)}\right)^{-\max\left(2,\frac{d}{2}\right)}\]
for some constant \(C_{Q_n,Q_m}>0\) that is independent of \(u\). This  follows by  the embedding \(H^n(z+[0,1]^d)\hookrightarrow W^{m,2\min\left(2,\frac{d}{d-2}\right)}(z+[0,1]^d)\) for every \(z\in\Z^d\).

Consider now \(u\) a minimizer for this problem (normalized by \(\int_{\R^d}Q_m(\nabla^m u)=1\)), and
\[\lambda:=\int_{\R^d}Q_n(\nabla^nu).\]
Then by the scaling properties of \(\Lambda_1[\ \cdot\ ;Q_n/Q_m]\), \(u\) is also a minimizer of
\begin{equation*}\label{eq_penalized}
v\in H^n(\R^d,\R^p)\setminus\{0\}\mapsto \frac{\int_{\mathbb{R} ^d}Q_n(\nabla^nv)}{\int_{\mathbb{R} ^d}Q_m(\nabla^m v)}+\frac{2(n-m)}{d}\lambda|\{v\neq 0\}|
\end{equation*}
In all that follows, we fix
\[\ov{r}:=\min\left(\left(\frac{\Lambda_1[\B_1;Q_n/Q_m]}{2\lambda}\right)^\frac{1}{2(n-m)},|\B_1|^{-\frac{1}{d}}\right)\]
so that in every ball \(B\) of radius \(r\in (0,\ov{r})\), for every \(v\in H^n_0(B)\), we have
\begin{equation}\label{eq_aux_consequencerbar}
\int_{B}\left(Q_n(\nabla^n v)-\lambda Q_m(\nabla^m v)\right)\geq \frac{1}{2}\int_{B}Q_n(\nabla^n v).
\end{equation}
In particular, for any such ball, there exists a unique minimizer to the energy
\[v\in u+H_0^n(B,\R^p)\mapsto \int_{B}\left(Q_n(\nabla^n v)-\lambda Q_m(\nabla^m v)\right)\]
which we denote \(u_{B}\). \(u_B\) is not identically zero since \(|B|<1=|\{u\neq 0\}|\) (by the definition of $\ov{r}$). Taking \(u_B\) as a competitor we obtain
\[\lambda+\frac{2(n-m)}{d}\lambda|\{u\neq 0\}| \leq \frac{\lambda+\int_{B}\left(Q_n(\nabla^n u_B)-Q_n(\nabla^n u)\right)}{1+\int_{B}\left(Q_m(\nabla^m u_B)-Q_m(\nabla^m u)\right)}+\frac{2(n-m)}{d}\lambda|\{u_B\neq 0\}|\]
which simplifies to
\begin{align*}
&\int_{B}\left(Q_n(\nabla^n u)-\lambda Q_m(\nabla^m u)-Q_n(\nabla^n u_B)+\lambda Q_m(\nabla^m u_B)\right)\\
&\leq \left(1+\int_{B}\left(Q_m(\nabla^m u_B)-Q_m(\nabla^m u)\right)\right)\frac{2(n-m)}{d}\lambda|B\cap\{u= 0\}|
\end{align*}
The left-hand side is estimated as follows:
\begin{align*}
\int_{B}\Bigl(Q_n(\nabla^n u)&-Q_n(\nabla^n u_B)-\lambda Q_m(\nabla^m u)+\lambda Q_m(\nabla^m u_B)\Bigr)
\\
&=\int_{B}\left(Q_n(\nabla^n(u-u_B))-\lambda Q_m(\nabla^m (u-u_B))\right)&\text{ by minimality of }u_B\\
&\geq \frac{1}{2}\int_{B}Q_n(\nabla^n(u-u_B))&\text{ by eq. }\eqref{eq_aux_consequencerbar}\\
&\geq \frac{\lambda_Q}{2}\int_{B}|\nabla^n(u-u_B)|^2&\text{ by eq. }\eqref{eq_legendrehadamard}
\end{align*}
We obtain
\[\int_{B}|\nabla^n(u-u_B)|^2\leq \left(1+\int_{B}\left(Q_m(\nabla^m u_B)-Q_m(\nabla^m u)\right)\right)\frac{4(n-m)\lambda}{d\lambda_Q}|B\cap\{u= 0\}|\]
Finally, we bound the factor on the right-hand side by
\begin{align*}
\int_{B}\left(Q_m(\nabla^m u_B)-Q_m(\nabla^m u)\right)&\leq C_Q\int_{B}\left(|\nabla^m u_B|^2+|\nabla^m u|^2\right)\text{ for some }C_Q>0\\
&\leq C_Q\int_{B}\left(2|\nabla^m (u-u_B)|^2+3|\nabla^m u|^2\right)\\
&\leq C_Q\int_{B}\left(2\Lambda^{-1} r^{2(n-m)}|\nabla^n (u-u_B)|^2+3|\nabla^m u|^2\right)
\end{align*}
where \(\Lambda:=\inf_{v\in H^n_0(\B_1)\setminus\{0\}}\frac{\int_{\B_1}|\nabla^n v|^2}{\int_{\B_1}|\nabla^mv|^2}\). Thus, for sufficiently small radius \(r\in (0,\ov{r})\), we obtain that for some constant \(C>0\), for any sufficiently small ball \(B\subset\R^d\),
\[\int_{B}|\nabla^n(u-u_B)|^2\leq C|B\cap\{u= 0\}|.\]
We let \(\phi=\frac{\nabla^n u}{\sqrt{C}}\), and for any ball \(B\) of sufficiently small radius we let \(\phi_B=\frac{\nabla^n u_B}{\sqrt{C}}\). Then \(\phi\) and the family of functions \((\phi_B)_B\) satisfy the hypothesis~\eqref{eq_hyp1} of Theorem~\ref{th_main}. Moreover, \(\phi_B\) satisfies the hypothesis~\eqref{eq_hyp2} by interior elliptic regularity.

By application of Theorem~\ref{th_main}, we conclude that \(u\in W^{n,\infty}(\R^d,\R^p)\). In particular, we may define the open set \(\Om=\cup_{j=0}^{n-1}\{\nabla^j u\neq 0\}\), so that \(u\in H^n_0(\Om,\R^p)\) (by~\cite[Th. 9.1.3]{AH12}), \(\Lambda_1[\Om;Q_n/Q_m]=\lambda\), and \(\Om\) is a minimizer.
Finally,  boundedness of \(\Om\) is a consequence of the non-degeneracy.

\end{proof}
Similarly, we have the existence and optimal regularity for minimizers of the Stokes eigenvalue. For any \(\Om\in\mathcal{A}\), consider
\[S_1[\Om]=\inf_{u\in H^1_{0,\mathrm{div}}(\Om,\R^d)\setminus\{0\}}\frac{\int_{\Om}|\nabla u|^2}{\int_{\Om}|u|^2}\]
the first eigenvalue of the Stokes operator. There is a subtle issue with the definition of \(S_1[\Om]\): we could also define the first eigenvalue of the Stokes operator as
\[S_1'[\Om]:=\inf\left\{ \frac{\int_{\Om}|\nabla u|^2}{\int_{\Om}|u|^2},\ u\in \mathcal{C}^\infty_c(\Om,\R^d)\setminus\{0\}: \mathrm{div}(u)=0\right\}.\]
However, it is unclear whether \(S_1[\Om]\) and \(S_1'[\Om]\) coincide. More precisely we clearly have \(S_1[\Om]\leq S_1'[\Om]\), and to obtain the opposite inequality we need to prove that any element of \(H^1_{0,\mathrm{div}}(\Om,\R^d)\) may be approached by a sequence in \(H^1_{0,\mathrm{div}}(\Om,\R^d)\cap\mathcal{C}^\infty_c(\Om,\R^d)\), which is not known for a general domain of finite measure. We refer to the recent work~\cite{DelNinWu24} for an extensive survey of this issue with several regularity conditions on the boundary that imply the equality \(S_1[\Om]= S_1'[\Om]\).
\begin{proposition}\label{prop_eigenvalue_stokes}
The minimization problem \(\inf_{\Om\in\mathcal{A}}S_1[\Om]\) reaches its minimum. Moreover any first eigenfunction \(u\) associated to some minimizer belongs to \(W^{1,\infty}(\R^d,\R^d)\).
\end{proposition}
It is not known whether an optimal domain is necessarily bounded, due to the absence of a non-degeneracy property.

%To extend this result for minimizers of \(S_1'\), a possibility would be to prove that a minimizer of \(S_1\) is sufficiently regular so that \(S_1\) coincides with \(S_1'\).

The minimization of \(S_1\) was extensively studied in~\cite{HMP24}, in particular the authors prove in~\cite[Th. 1]{HMP24} that a minimizer among quasi-open sets exists, and such a minimizer is open by~\cite[Th. 1]{F25}. It is then proved that the ball is a local minimizer (among sufficiently small, smooth perturbations of the ball) in two dimensions, but not in three dimensions; in fact by~\cite[Cor. 1]{HMP24}, any sufficiently regular boundary component of a minimizer in \(\R^3\) must be homeomorphic to a torus. We remind that in the drag minimization studied in~\cite{P73,B74}, the boundary has singular points, so it is unclear whether one can expect the full regularity for minimizers of \(S_1\).
\begin{proof}[Proof of Proposition~\ref{prop_eigenvalue_stokes}]
As previously, we relax the problem into
\[\inf_{\Om\in\mathcal{A}}S_1[\Om]\geq \inf\left\{\frac{\int_{\R^d}|\nabla u|^2}{\int_{\R^d}|u|^2},\ u\in H^1_{\mathrm{div}}(\R^d,\R^d)\setminus\{0\}: |\{u\neq 0\}|\leq 1\right\}\]
The existence of a minimizer \(u\) for the right-hand side is proved in~\cite[Th. 1]{HMP24} by concentration-compactness methods. We assume \(\int_{\R^d}|u|^2=1\) and let \(\lambda=\int_{\R^d}|\nabla u|^2\).
As in the proof of Proposition~\ref{prop_eigenvalue}, we consider for any sufficiently small ball \(B\) the minimizer of
\[v\in u+H^{1}_{0,\mathrm{div}}(B,\R^d)\mapsto\int_{B}\left(|\nabla v|^2-\lambda |v|^2\right),\]
that we denote \(u_B\). Considering \(u_B\) as a competitor, we obtain (as previously) that, provided the radius of \(B\) is sufficiently small,
\[\int_{B}|\nabla (u-u_B)|^2\leq C|B\cap\{u=0\}|.\]
for some constant \(C>0\). Theorem~\ref{th_main} then applies to some scalar factor of \(\nabla u\), \((\nabla u_B)_{B\subset\B_1}\); we obtain that \(u\in W^{1,\infty}(\R^d,\R^d)\). We conclude by letting \(\Om=\{u\neq 0\}\), which is an open set such that \(u\in H^1_0(\Om,\R^d)\), so \(\lambda=S_1[\Om]\) and \(\Om\) is a minimizer.
\end{proof}

We now give some case of applications for the Lipschitz regularity of inner variation solutions, to several questions of spectral geometry and overdetermined problems.

\begin{proposition}\label{prop_overdetermined}
Assume \(d=2\). Let \(u\in H^1(\B_1,\R)\), \(f\in\mathcal{C}^\alpha_\loc(\R,\R)\) for some \(\alpha\in (0,1)\), such that \(f\circ u\in L^p(\B_1)\) for some \(p\in (1,+\infty)\). Assume that for any \(\xi\in\mathcal{C}^\infty_c(\B_1,\R^2)\) we have
\[\left.\frac{d}{dt}\right|_{t=0}\int_{\B_1}\left(|\nabla (u\circ\xi_t)|^2-f(u\circ\xi_t)+\chi_{u\circ\xi_t\neq 0}\right)=0\]
where \(\xi_t=\mathrm{id}+t\xi\). Then \(u\in W^{1,\infty}_\loc(\B_1,\R)\).
\end{proposition}

This applies to the following cases:
\begin{itemize}
\item For any open set \(\Om\subset\R^2\) with finite measure, denote \(w_\Om\in H^1_0(\Om)\) the solution of \[-\Delta w_\Om=1\text{ in }\Om.\]
Write \(T[\Om]=\int_{\Om}|\nabla w_\Om|^2\) the \emph{torsional rigidity} of \(\Om\), and assume that for any \(\xi\in\mathcal{C}^\infty_c(\R^2,\R^2)\), we have
\begin{equation}\label{eq_innervariation_torsion}
\left.\frac{d}{dt}\right|_{t=0}\Big(T\left[(\mathrm{id}+t\xi)(\Om)\right]-c |(\mathrm{id}+t\xi)(\Om)|\Big)=0.
\end{equation}
for some constant \(c>0\). Then \(w_\Om\in W^{1,\infty}(\Om,\R)\) and \(\Om\) is bounded. Indeed the first conclusion is a consequence of Proposition~\ref{prop_overdetermined} with \(f(u):=2u\), and the second conclusion is a consequence of Proposition~\ref{prop_crit_nondeg}, which extends naturally to this case. This problem is related to Serrin's theorem stating that any sufficiently smooth bounded domain \(\Om\) such that \(|\nabla w_\Om|\) is constant on \(\partial\Om\), must be some disjoint union of balls.
This result was recently generalized in~\cite{FZ25} for bounded, finite perimeter, indecomposable domains \(\Om\) satisfying the distributional condition
\begin{equation}\label{eq_outervariation}
-\Delta w_\Om=1_\Om-c\mathcal{H}^{1}\lfloor \partial^*\Om.
\end{equation}
for some $c>0$, and the (one-sided) Ahlfors regularity condition for some $A>0$:
\begin{equation}\label{eq_Ahlfors}
\mathscr{H}^1(\partial^*\Om\cap\B_{z,r})\leq A r\text{ for }\mathscr{H}^1\text{-a.e. } z\in\partial^*\Om,\ r\in (0,1).
\end{equation}
On the contrary, in~\cite[Th. 1.2]{Z26} it was proved that there exist bounded, finite perimeter, indecomposable, open sets \(\Om\subset\R^2\) that satisfy the condition~\eqref{eq_outervariation}, which are not balls.
Using Proposition~\ref{prop_overdetermined}, any finite perimeter, indecomposable domain \(\Om\subset\R^2\) satisfying both \eqref{eq_outervariation} and \eqref{eq_innervariation_torsion} is automatically bounded and satisfies the condition \eqref{eq_Ahlfors} (as may be observed by testing \eqref{eq_outervariation} against some \(\varphi\in\mathcal{C}^\infty_c(\B_{z,2r},\R)\) such that \(\varphi|_{\B_{z,r}}\equiv 1\) and \(|\nabla\varphi|\lesssim r^{-1}\), and using the Lipschitz bound on \(w_\Om\)), thus \(\Om\) must be a ball by~\cite[Th. 1.3]{FZ25}. This proves that the counterexamples built in~\cite[Th. 1.2]{Z26} are in fact not critical points of the torsional rigidity \(T(\Om)\).
\item For any open set \(\Om\subset\R^2\) with finite measure, denote \((\lambda_k[\Om])_{k\geq 1}\) the non-decreasing sequence of eigenvalues of \(-\Delta\) in \(\Om\) with Dirichlet boundary conditions on \(\partial\Om\). Assume that for some \(k\in\N^*\), \(\lambda_k[\Om]\) is a simple eigenvalue and for any \(\xi\in\mathcal{C}^\infty_c(\R^2,\R^2)\), we have
\[\left.\frac{d}{dt}\right|_{t=0}\Big(\lambda_k\left[(\mathrm{id}+t\xi)(\Om)\right]\cdot|(\mathrm{id}+t\xi)(\Om)|\Big)=0.\]
Then \(u_k\in W^{1,\infty}(\R^2,\R)\) and \(\Om\) is bounded. Indeed the first conclusion is a consequence of Proposition~\ref{prop_overdetermined} with \(f(u):=\lambda_k[\Om]u^2\), and the second conclusion is a consequence of Proposition~\ref{prop_crit_nondeg}.
% This overdetermined problem is related to the Schiffer conjecture, which may be seen as a generalization of the Serrin problem, which was studied in~\cite{CH99}. It is unknown whether this condition forces \(\Omega\) to be a ball, even for \(k=1\) due to the absence of regularity conditions and the possibility of two-phase points.

\end{itemize}

\begin{proof}[Proof of Proposition \ref{prop_overdetermined}]
Following the computations of Lemma~\ref{lem_equ}, this is equivalent to the weak equation:
\[4\partial_{\ov{z}}\Ho{u}+\partial_{z}\left(f(u)+\chi_{u= 0}\right)=0\]
in \(\B_1\). Thus the function
\[4\Ho{u}+\mathcal{B}\left[\left(f(u)+\chi_{u= 0}\right)\chi_{\B_1}\right]\]
is harmonic in \(\B_1\). Since \(f(u)\in L^p(\B_1)\) and \(\Vert\mathcal{B}\Vert_{L^p(\R^2)\to L^p(\R^2)}\) is finite (see for instance~\cite[Th. 4.5.3]{AIM09}), then \(\Ho{u}\in L^p(\B_1,\C)\), so \(u\in\mathcal{C}^{1-\frac{1}{p}}(\B_1,\R)\). By our hypothesis on \(f\) we have \(f(u)\in\mathcal{C}^{\left(1-\frac{1}{p}\right)\alpha}_\loc(\B_1,\R)\). By the boundedness of \(\mathcal{B}\) on \(\mathcal{C}^{\left(1-\frac{1}{p}\right)\alpha}(\R^2)\) (which is a consequence of classical Calderón--Zygmund theory, see also~\cite[Th. 4.7.1]{AIM09}), this implies that \(\mathcal{B}[f(u)\chi_{\B_1}]\) belongs to \(L^\infty_\loc(\B_1)\).\bigbreak
Let now
\[\Phi:=4\Ho{u}+\mathcal{B}\left[f(u)\chi_{\B_1}\right]\chi_{\B_1\cap\{u\neq 0\}}\]
and for any \(\B_{z,r}\subset\B_1\), let
\[\Phi_{\B_{z,r}}:=\Phi+\mathcal{B}\left[\chi_{\B_{z,r}\cap\{u=0\}}\right]+\mathcal{B}\left[f(u)\chi_{\B_1}\right]\chi_{\B_1\cap\{u= 0\}}\]
\(\Phi_{\B_{z,r}}\) is harmonic in \(\B_{z,r}\) by Lemma~\ref{lem_B}, so \(\Phi_{\B_{z,r}}\) satisfies the hypothesis~\eqref{eq_hyp2} for some universal constant \(K>0\). Moreover, assuming that \(\B_{z,r}\subset\B_\rho\) for some \(\rho\in (0,1)\), we have
\begin{align*}
\int_{\B_{z,r}}\left|\Phi_{\B_{z,r}}-\Phi\right|^2&\lesssim \int_{\B_{z,r}}\left(\left|\mathcal{B}\left[\chi_{\B_{z,r}\cap\{u=0\}}\right]\right|^2+\chi_{\B_{z,r}\cap\{u=0\}}\left|\mathcal{B}\left[f(u)\chi_{\B_1}\right]\right|^2\right)\\
&\lesssim \left(1+\left\Vert \mathcal{B}\left[f(u)\chi_{\B_1}\right]\right\Vert_{L^\infty(\B_\rho)}^2\right)|\B_{z,r}\cap\{u=0\}|.
\end{align*}
Thus the hypothesis~\eqref{eq_hyp1} is satisfied by $\Phi$ in \(\B_\rho\) up to some scalar factor depending on \(\rho\). Thus by application of Theorem~\ref{th_main}, we obtain \(\Phi\in L^\infty_\loc(\B_1)\). Since \(\mathcal{B}\left[f(u)\chi_{\B_1}\right]\in L^\infty_\loc(\B_1)\), we conclude
\[4\Ho{u}=\Phi-\mathcal{B}\left[f(u)\chi_{\B_1}\right]\chi_{\B_1\cap\{u\neq 0\}}\in L^\infty_\loc(\B_1).\]
Since \(|\nabla u|^2=4|\Ho{u}|\), this concludes the proof.
\end{proof}

\section*{Acknowledgment}
GDP's research is funded by the European Research Council (ERC) through CoG 101169953 ``RISE''.\footnote{Views and opinions expressed are however those of the authors only and do not necessarily reflect those of the European Union or the European Research Council.}
MN's research is supported by the project ANR STOIQUES (ANR-24-CE40-2216) financed by the French Agence Nationale de la Recherche (ANR). No AI tools were used in the preparation of this manuscript.

% \bibliographystyle{plain}
% \bibliography{biblio}

\printbibliography

\end{document}

% Local Variables:
% jinx-local-words: "vectorial"
% End: